\documentclass[a4paper,11pt,reqno]{amsart}
\usepackage{enumerate}
\usepackage{enumitem}
\usepackage{amsfonts,amssymb,amsmath,amsthm}
\usepackage{epsfig}
\usepackage{graphics}
\usepackage[normalem]{ulem}
\usepackage{xcolor}
\usepackage{comment}
\usepackage{stmaryrd} 
\usepackage{url}
\usepackage{overpic}
\usepackage{mathrsfs}
\usepackage{tikz-cd}
\usepackage[british]{babel}

\usepackage[a4paper,top=3cm,bottom=3cm,left=3cm,right=3cm,marginparwidth=2cm]{geometry}

\usepackage[colorlinks=true, allcolors=blue]{hyperref}
\usepackage{mathtools}

\usepackage[draft]{fixme}
\FXRegisterAuthor{jr}{ajr}{Reppekus}
\FXRegisterAuthor{hp}{ahp}{Peters}
\FXRegisterAuthor{gr}{agr}{Regts}

\renewcommand{\Pr}{\mathbb{P}}

\theoremstyle{plain}
\newtheorem*{main}{Main Theorem}
\newtheorem{thm}{Theorem}
\newtheorem{lem}[thm]{Lemma}
\newtheorem{prop}[thm]{Proposition}
\newtheorem{cor}[thm]{Corollary}

\theoremstyle{definition}
\newtheorem{defn}[thm]{Definition}

\theoremstyle{remark}
\newtheorem{rem}[thm]{Remark}
\newtheorem{note}[thm]{Note}

\DeclarePairedDelimiter{\paren}{\lparen}{\rparen}

\DeclarePairedDelimiter{\abs}{\lvert}{\rvert}
\DeclarePairedDelimiter{\norm}{\lVert}{\rVert}

\title{Decay of correlations and zeros for the hard-core model}
\author{Han Peters}
\address{Korteweg de Vries Institute for Mathematics, University of Amsterdam
P.O. Box 94248 1090 GE, Amsterdam, The Netherlands}
\email{h.peters@uva.nl}
\author{Guus Regts}
\email{g.regts@uva.nl}
\author{Josias Reppekus}
\email{j.reppekus@uva.nl}

\begin{document}
\begin{abstract}
    In a recent paper the last author proved that absence of complex zeros of the partition function of the hard-core model near a parameter $\lambda>0$ implies a form of correlation decay called strong spacial mixing. In this paper we investigate the reverse implication.
    
    We introduce a strengthening of strong spatial mixing that we call \emph{very strong spatial mixing} (VSSM).
    Our main result is that if VSSM holds at a parameter $\lambda>0$ for a family of graphs, this implies that the partition function has no zeros near that parameter for each graph in the family. We also demonstrate that a closely related variant of very strong spatial mixing does not imply zero-freeness.
    As a consequence of our main result, we moreover obtain that VSSM implies spectral independence. 
    
    Our proof relies on transforming the problem to the analysis of an induced non-autonomous dynamical system given by M\"obius transformations.
\end{abstract}
\maketitle

\section{Introduction} 
The \emph{hard-core model} at \emph{fugacity} $\lambda>0$ on a finite graph $G=(V,E)$ is the probability measure on the collection of independent sets $\mathcal{I}_G$ of $G$ (i.e. $\mathcal{I}_G:=\{I\subseteq V \mid \text{ for all } u,v\in I: \{u,v\}\notin E\}$) defined by 
\[
\mu_{G,\lambda}(I):=\frac{\lambda^{|I|}}{\sum_{J\in\mathcal{I}_G}\lambda^{|J|}},
\]
where we often just write $\mu$ instead of $\mu_{G,\lambda}$ in case $\lambda$ and $G$ are clear from the context.
Here 
\[
Z_G(\lambda):=\sum_{J\in\mathcal{I}_G}\lambda^{|J|},
\]
is the \emph{partition function} of the hard-core model, also known as the independence polynomial of $G$.

The hard-core model and its partition function have been studied and used from a variety of angles. In  statistical physics, the hard-core model or hard-core lattice gas model, is a discrete model for the hard-sphere model, in which it is assumed that each atom of a gas takes up a fixed space in a given volume, a \emph{core}, which cannot intersect the core of any other atom. This is for example used to model adsorption of molecules on to a crystal surface, see e.g.~\cite{absorption}. 

In extremal combinatorics the hard-core model has been used to study regular graphs with the most independent sets~\cite{Perkinsoccupancyfraction, Perkinsavgsizeindset} and is related to the Lov\'asz Local Lemma~\cite{ScottSokal05}.
The question of the existence of efficient algorithms to (approximately) compute $Z_G(\lambda)$ has received much attention in theoretical computer science~\cite{Weitz2006CountingIndependentSetsuptotheTreeThreshold,SlySun2014CountinginTwoSpinModelsondRegularGraphs,GalanisStefankovicVigoda2016InapproximabilityofthePartitionFunctionfortheAntiferromagneticIsingandHardCoreModels,SinclairSrivastavaStefankovicYin2017SpatialMixingandtheConnectiveConstantOptimalBounds,Barvinok2016CombinatoricsandComplexityofPartitionFunctions,PatelRegts17,BezakovaGalanisGoldbergStefankovic2020InapproximabilityoftheIndependentSetPolynomialintheComplexPlane,deBoeretalchaoticratios}.

Two of the main approaches in the design of efficient deterministic\footnote{There are also efficient randomized algorithms based on Markov chains, see Subsection~\ref{ssec:spectral ind} of the present paper.} algorithms for approximating the partition function $Z_G(\lambda)$ are based on seemingly distinct properties of the hard-core model and connecting these forms the main motivation for the present paper.

One is the correlation decay approach pioneered by Weitz~\cite{Weitz2006CountingIndependentSetsuptotheTreeThreshold} about twenty years ago, and allows to design efficient algorithms provided the model satisfies a strong form of correlation decay, which roughly says that the influence on a given vertex of vertices far away in the graphs decays exponentially fast. We will make this precise below.
The other approach is the interpolation method of Barvinok~\cite{Barvinok2016CombinatoricsandComplexityofPartitionFunctions} combined with an algorithm due to Patel and the second author~\cite{PatelRegts17}.
This approach yields efficient algorithms provided the complex zeros of the polynomial $Z_G(x)$ do not come arbitrarily close to the interval $[0,\lambda]$.
Notably both approaches are inspired by statistical physics and operate in the regime where the model does not exhibit phase transitions, in the sense of uniqueness of the Gibbs measure for the first approach and in the Lee-Yang~\cite{LeeYang1} sense for the second approach.
Additionally, they yield the same algorithmic results for the family  $\mathcal{G}_\Delta$ of all graphs of a given maximum degree $\Delta\in \mathbb{N}$. More precisely, assuming $\Delta\geq 3$, for $0\leq \lambda<\lambda_c(\Delta)$, where
\[
\lambda_c(\Delta):=\frac{(\Delta-1)^{\Delta-1}}{(\Delta-2)^\Delta},
\]
both approaches yield a fully polynomial time approximation algorithm to compute a relative $\varepsilon$-approximation to $Z_G(\lambda)$ for $G\in \mathcal{G}_\Delta$ while for for $\lambda>\lambda_c(\Delta)$ the problem of even computing a relative $2$-approximation to $Z_G(\lambda)$ is \textsc{NP}-hard~\cite{SlySun2014CountinginTwoSpinModelsondRegularGraphs,GalanisStefankovicVigoda2016InapproximabilityofthePartitionFunctionfortheAntiferromagneticIsingandHardCoreModels}.

It is therefore natural question to ask how the underlying properties of correlation decay and absence of complex zeros relate. 
For certain structured families of graphs such as lattices this question has been answered in a very strong sense by Dobrushin and Schlossman~\cite{DS84,DS87} around forty years ago.
They provided more than a dozen properties of the model and showed that they are equivalent at a given $\lambda>0$. Moreover, their work is not restricted to the hard-core model, but applies to many local interaction models.
However, their proof crucially relied on the structural aspect of lattices that in these graphs the collection of vertices at distance $r$ of a given vertex only grows polynomially in $r$.
Since for a typical graph in $\mathcal{G}_\Delta$ this growth is exponential, the question of the connection between absence of zeros and correlation decay for the hard-core model is not answered by the work of Dobrushin and Schlossman~\cite{DS84,DS87}.

The broader question of the connection between correlation decay and absence of zeros on more general families of graphs (of bounded degree) has gained interest in the past few years.
For example, Liu, Sinclair, and Srivastava~\cite{LSSCorrelationdecayandcomplexzeros} recognized that a successful approach for proving decay of correlation for models such as the hard-core model at a parameter could also be used to show absence of zeros of the corresponding partition function near the parameter.
Recently, the second author~\cite{Regts2023AbsenceofZerosImpliesStrongSpatialMixing}, improving on work of Gamarnik~\cite{Gamarnikcorrelationandzeros}, showed that absence of complex zeros near an interval of the form $[0,\lambda]$ implies a strong form of correlation decay, called \emph{strong spatial mixing}, for parameters in that interval (for any given family of bounded degree graphs).

In the present paper, we contribute to this line of research by introducing a strengthening of the notion of strong spatial mixing that we coin \emph{very strong spatial mixing} (VSSM) and show that this property implies zero-freeness of the partition function near the parameter $\lambda$.
Informally, a family of graphs that is closed under taking induced subgraphs satisfies VSSM at $\lambda>0$ if the family of self-avoiding walk trees of the graphs in the family satisfies weak spatial mixing at $\lambda$. 
See Definition~\ref{def:VSSM} below for the precise definition.
We note that for $\lambda>0$ strong spatial mixing on the infinite $\Delta$-regular tree is equivalent to VSSM on for the class of graphs of maximum degree at most $\Delta$, linking our definition of VSSM to the more standard notion of strong spatial mixing.
Our main result is the following:

\begin{main}
    Let $\mathcal{G}$ be a family of bounded degree graphs closed under taking induced subgraphs and let $\lambda^\star>0$. If the hard-core model on $\mathcal{G}$ satisfies VSSM at $\lambda^\star$, then there exists an open set $U\subseteq \mathbb{C}$ containing $\lambda^\star$ such that $Z_G(\lambda)\neq 0$ for all $\lambda \in U$ and any graph $G\in \mathcal{G}$.
\end{main}

We note that for the family of graphs of maximum degree at most $\Delta$ our main result combined with the result from~\cite{Regts2023AbsenceofZerosImpliesStrongSpatialMixing} provides a clear explanation as to why the algorithmic approaches of Weitz and Barvinok essentially yield the same results. (In fact, in~\cite{bencs2025barvinok} it was observed that also the underlying algorithms are nearly the same.)

In the following Section~\ref{subsec:DefsAndResults}, we formally introduce the notion of very strong spatial mixing, state a multi-variate version of the main result, and discuss results demonstrating that the main result no longer holds under certain relaxations of the very strong spatial mixing assumption.
In Section~\ref{ssec:spectral ind}, we derive from our main result that very strong spatial mixing implies spectral independence, a property relevant in the analysis of the mixing time of the Glauber dynamics for the hard-core model.
Finally, in Section~\ref{ssec:proof overview} we give a rough outline of our proof of the main result and an overview of the remainder of the paper.

\subsection{Definitions and main result}\label{subsec:DefsAndResults}
We start by introducing some standard terminology.

For a graph $G=(V,E)$ and a vertex $v_{0}\in V$ we have
\[
Z_{G}^{v_{0}\textnormal{ in}}(\lambda) \coloneqq \sum_{\substack{I \in \mathcal{I}_G\\
        v_{0}\in I}} \lambda^{|I|} = \lambda Z_{G\setminus N(v_{0})}(\lambda)
\]
and 
\[
Z_{G}^{v_{0}\textnormal{ out}}(\lambda) \coloneqq \sum_{\substack{I \in \mathcal{I}_G\\
        v_{0}\notin I}} \lambda^{|I|} = Z_{G-v_{0}}(\lambda),
\]
where $N(v_{0})$ denotes the set of $v_{0}$ and all its neighbours,
$G\setminus V'$ denotes the induced subgraph of $G$ with vertex
set $V\setminus V'$, and $G-v$ the induced subgraph with vertex
set $V\setminus\{v\}$. 
This allows to write
\begin{equation}\label{eq:fundamental recurrence}
    Z_G(\lambda)= \lambda Z_{G\setminus N(v_{0})}(\lambda)+ Z_{G-v_{0}}(\lambda).
\end{equation}
We define the rational function
\[
R_{G,v_{0}}(\lambda)\coloneqq\frac{Z_{G}^{v_{0}\textnormal{ in}}(\lambda)}{Z_{G}^{v_{0}\textnormal{ out}}(\lambda)}
\]
and refer to this as the \emph{occupation ratio at $v_0$} (we will often write ratio for short).
We note that in case $\lambda>0$ we can interpret the ratio as 
\begin{equation}\label{eq:prob interpretation ratio}
    R_{G,v_0}(\lambda)=\frac{\Pr_{\mu}[v_0\in \bf I]}{\Pr_{\mu}[v_0\notin\bf I]}.
\end{equation}

Let $G=(V,E)$ be a graph. A \emph{boundary condition} on $G$ for
the hard-core model is a map 
\[
\sigma:V_{\sigma}\to\{0,1\}
\]
on a subset $V_{\sigma}\subset V$ such that $V_{\sigma, \mathrm{in}}:=\sigma^{-1}(1)\subseteq V$ is an independent set of $G$. 
We mostly work with boundary conditions at a fixed distance of a given vertex $v$. To this end let $\mathbb N$ denote the set of positive integers $\{1,2,\ldots\}$. For a graph $G=(V,E)$, a vertex $v\in V$, and $\ell\in \mathbb{N}$, let $S_G(v,\ell)$ denote the collection of all vertices of $G$ at distance $\ell$ from $v$ in $G$.

For a boundary condition $\sigma$ on $G$, we define the conditional partition function $Z_{(G,\sigma)}(\lambda)$ as
\[
Z_{(G,\sigma)}(\lambda):=\sum_{\substack{I\subseteq V\setminus V_\sigma\\ I\cup V_{\sigma, \mathrm{in}} \in \mathcal{I}_G}} \lambda^{|I|}.
\]
For $v\notin V_\sigma$ we subsequently define the conditional ratio 
\[
R_{(G,\sigma),v}(\lambda)\coloneqq\frac{Z^{v \textnormal{ in}}_{(G,\sigma)}(\lambda)}{Z^{v \textnormal{ out}}_{(G,\sigma)}(\lambda)},
\]
where 
\[
Z_{(G,\sigma)}^{v\textnormal{ in}}(\lambda)=\sum_{\substack{v\in I\subseteq V\setminus V_{\sigma}\\
        I\cup V_{\sigma,\mathrm{in}}\in \mathcal{I}_G
    }
}\lambda^{|I|}\quad\text{and}\quad Z_{(G,\sigma)}^{v\textnormal{ out}}(\lambda)=\sum_{\substack{v\notin I\subseteq V\setminus V_{\sigma}\\
        I\cup V_{\sigma,\mathrm{in}}\in \mathcal{I}_G
    }
}\lambda^{|I|},
\]
and note that as in~\eqref{eq:prob interpretation ratio} for $\lambda>0$ we have 
\begin{equation}\label{eq:prob interpretation ratio with boundary}
    R_{(G,\sigma),v}(\lambda)=\frac{\Pr_{\mu}[v_0\in {\bf I}\mid  {\bf I}\sim\sigma]}{\Pr_{\mu}[v_0\notin{\bf I}\mid  {\bf I}\sim\sigma]},
\end{equation}
where the notation $ {\bf I}\sim\sigma$ stands for the event that the random independent set $\bf I$ agrees with $V_{\sigma,\mathrm{in}}$ on $V_\sigma$.


To define very strong spatial mixing we will need the concept of the tree of self avoiding walks, as introduced by Scott and Sokal~\cite{ScottSokal05} and Weitz~\cite{Weitz2006CountingIndependentSetsuptotheTreeThreshold}. 
See also Bencs~\cite{Bencsreal-rooted}. Informally, the vertices of this tree correspond to paths starting at a given vertex $v$ of the graph and each path is connected with an edge to its single extensions. The following equivalent and more formal definition will be useful in later proofs.

\begin{defn}[The tree of self avoiding walks] \label{def:TSAW}
    
    
    Given a graph $G=(V,E)$ and a vertex $v$, the \emph{tree of self avoiding walks} is the tree $T_{\mathrm{SAW}(G,v)}$ obtained from $(G,v)$ as follows.
    Fix a total ordering of the vertices of $G$.
    We define the tree $T_{\mathrm{SAW}(G,v)}$ recursively. 
    If the component of $G$ containing $v$ consists of a single vertex $v$, then $T_{\mathrm{SAW}(G,v)}=\{v\}$.
    In case the component of $G$ containing $v$ consists of more than $1$ vertex, let $v_1,\ldots,v_d$ denote the neighbours of $v$ in $G$.
    Let for $i=1\ldots d$, $G_i$ be the graph obtained from $G$ by removing $v,v_1,\ldots,v_{i-1}$ from $G$.
    Denote for $i=1,\ldots, d$, $T_i=T_{\mathrm{SAW}(G_{i},v_{i})}$.
    Then $T_{\mathrm{SAW}(G,v)}$ is defined as the tree obtained from the disjoint union of the $T_i$ adding a vertex $v$ connecting it to each of the $v_1,\ldots,v_d$. 
\end{defn}

\begin{rem}
    The tree of self avoiding walks as defined above is closer to Bencs's definition~\cite{Bencsreal-rooted} of the stable path tree and is somewhat more compact than Weitz's original definition~\cite{Weitz2006CountingIndependentSetsuptotheTreeThreshold}, which involves boundary conditions. 
    The use of boundary conditions is necessary in case one wants to work with other models than the hard-core model such as the Ising model.
    
    Bencs~\cite{Bencsreal-rooted} moreover gives other variants of the stable path tree. Each of these would suffices for our purposes as we will become clear later. 
    For clarity we just stick to the definition above.
\end{rem}

The next lemma indicates the usefulness of the tree of self avoiding walks.
\begin{lem}[Weitz~\cite{Weitz2006CountingIndependentSetsuptotheTreeThreshold}]\label{lem:weitz}
    Let $G$ be a graph with vertex $v$. Then
    \[
    R_{G,v}(\lambda)=R_{T_{\mathrm{SAW}(G,v)},v}(\lambda).
    \]
\end{lem}
We will need a slightly stronger version of this result that we prove in Lemma~\ref{lem:ratio TSAW is ratio G}.
We can now define the notion of decay of correlations that we need.

\begin{defn}\label{def:VSSM}
    The hard-core model on a family $\mathcal{G}$ of bounded degree graphs satisfies
    \emph{very strong spatial mixing (VSSM)} at a parameter $\lambda>0$ if there
    exist constants $\alpha\in[0,1)$ and $C>0$ such that for all $G=(V,E)\in\mathcal{G}$, any induced subgraph $H$ of $G$ and $v\in V(H)$, and for all boundary conditions $\sigma,\tau:S_{T_{\mathrm{SAW}(H,v)}}(v,\ell)\to\{0,1\}$ we have
    \begin{equation}
        \abs{R_{(T_{\mathrm{SAW}(H,v)},\sigma,v}(\lambda)-R_{(T_{\mathrm{SAW}(H,v)},\tau,v}(\lambda)}\le C\cdot \alpha^{\ell}\quad \text{for all }\ell\in \mathbb{N}.
        \label{eq:VSSM-via-R}
    \end{equation}
    We refer to $\alpha$ as the \emph{rate of VSSM}.
\end{defn}
Some remarks are in order.
\begin{rem}
    We note that in~\cite{nair2007correlation} the term very strong spatial mixing was used to describe a different notion. We trust, however, that there will be no confusion.
\end{rem}

\begin{rem} 
    Typically notions of correlation decay are formulated in terms of the sensitivity of $\Pr_\mu[v\in {\bf I}\mid  {\bf I}\sim\sigma]$ on the boundary condition $\sigma$.
    For a graph of maximum degree at most $\Delta$ we have
    \begin{equation}\label{eq:lower bound prob v in}
        \Pr_{\mu}[v_0\in {\bf I}]\geq\frac{\lambda}{\lambda+(1+\lambda)^{\Delta+1}},
    \end{equation} 
    as follows by conditioning on the vertices of distance $2$ from $v$.
    This implies that if the maximum degree $\Delta$ of the family $\mathcal{G}$ is fixed in advance, replacing the conditional ratios in~\eqref{eq:VSSM-via-R} by the conditional probabilities does 
    not affect the definition of VSSM.
\end{rem}

\begin{rem}\label{rem:vssm=ssm for trees}
    Let us recall that the hard-core model on a family $\mathcal{G}$ of bounded degree graphs satisfies
    \emph{strong spatial mixing (SSM)} at a parameter $\lambda>0$ if there
    exist constants $\alpha\in[0,1)$ and $C>0$ such that for all $G=(V,E)\in\mathcal{G}$, any induced subgraph $H$ of $G$ and $v\in V(H)$, and for all boundary conditions $\sigma,\tau:S_H(v,\ell)\to\{0,1\}$ we have
    \begin{equation}
        \abs{R_{(H,\sigma),v}(\lambda)-R_{(H,\tau),v}(\lambda)}\le C\cdot \alpha^{\ell}\quad \text{for all }\ell\in \mathbb{N}.
        \label{eq:SSM-via-R}
    \end{equation}
    With the aid of Lemma~\ref{lem:weitz} and Corollary~\ref{cor:vssm =ssm for trees} below it can be shown that VSSM implies SSM and hence VSSM is indeed a strengthening of SSM.
    
    We note that in case $\mathcal{G}$ consists only of trees, the two notions are clearly equivalent since the tree of self avoiding walks of a rooted tree is just the tree itself. 
\end{rem}

Given a graph $G=(V,E)$ and a vector of variables $(\lambda_v)_{v\in V}$ we define the multivariate independence polynomial by
\[
Z_{G}(\lambda_v):=\sum_{I\in \mathcal{I}_G} \prod_{v\in I} \lambda_v.
\]
The following multivariate statement implies the main result, and is the result that we will actually prove.

\begin{thm}\label{thm:main}
    Let $\Delta\in \mathbb{N}$ and let $\lambda^\star>0$.
    Let $\mathcal{G}$ be a family of graphs of maximum degree at most $\Delta$. If the hard-core model on $\mathcal{G}$ satisfies very strong spatial mixing at $\lambda^\star$, then there exists $\varepsilon>0$ such that for each $G=(V,E)\in \mathcal{G}$ and $\Vec\lambda\in B(\lambda^\star,\varepsilon)^V$ we have $Z_G(\Vec \lambda)\neq 0$.
\end{thm}

Given the results in~\cite{Regts2023AbsenceofZerosImpliesStrongSpatialMixing} it is natural to ask whether a converse to Theorem~\ref{thm:main} holds.
We note that this generally is not the case as, by a result of Chudnovsksy and Seymour~\cite{CSclawfree}, the independence polynomial of a claw-free graph has only real (and hence negative) zeros. Combining this result with an example of Bencs~\cite[Proposition 3.5]{Bencsreal-rooted} that gives a family of claw-free graphs whose self-avoiding walk trees contain arbitrarily large binary trees as (rooted) subtrees and therefore these will not satisfy weak spatial mixing at $\lambda=\lambda_c(3)=4$. Additionally, by~\cite{Regts2023AbsenceofZerosImpliesStrongSpatialMixing} this also shows that SSM and VSSM are not equivalent in general.
We return to the question of whether a certain converse to Theorem~\ref{thm:main} holds in Section~\ref{sec:conclusion}. 

We now focus on a natural relaxation of the notion of VSSM.
Recall that Weitz's algorithmic approach works well on a family of bounded degree graphs provided the hard-core model satisfies VSSM.
In fact, for algorithmic purposes one can do with the following slightly weaker notion, which was utilized in~\cite{SinclairSrivastavaStefankovicYin2017SpatialMixingandtheConnectiveConstantOptimalBounds}.

\begin{defn}\label{def:VSSM at a distance}
    Let $\varphi \colon \mathbb{N}\to\mathbb{R}^{+}$ be a function with positive
    values. 
    The hard-core model on a family $\mathcal{G}$ of bounded degree graphs satisfies
    \emph{very strong spatial mixing from distance $\varphi$ ($\varphi$-VSSM)} at a parameter $\lambda>0$ if there
    exist constants $\alpha\in[0,1),C>0$ such that for all $G=(V,E)\in\mathcal{G}$ and any induced subgraph $H$ of $G$ and $v\in V(H)$, and for all boundary conditions $\sigma,\tau:S_{T_{\mathrm{SAW}(H,v)}}(v,\ell)\to\{0,1\}$ we have
    \begin{equation}
        \abs{R_{(T_{\mathrm{SAW}(H,v)},\sigma),v}(\lambda)-R_{(T_{\mathrm{SAW}(H,v)},\tau),v}(\lambda)}\le C\cdot \alpha^{\ell}\quad \text{for all }\ell\geq \varphi(|V|).
        \label{eq:VSSM-via-R at distance}
    \end{equation}
\end{defn}

For families of graphs of bounded maximum degree it is not hard to see that $\varphi$-VSSM with $\varphi(n)=\log(n)$ is sufficient for algorithmic purposes using Weitz' approach. See~\cite{SinclairSrivastavaStefankovicYin2017SpatialMixingandtheConnectiveConstantOptimalBounds} for extensions to graphs of unbounded degree under additional conditions on the so-called connective constant.

Our next result shows that VSSM at distance $\varphi$ does no longer necessarily imply zero-freeness provided $\varphi$ is unbounded. This suggests that Barvinok's approach may not be applicable in all settings where Weitz's approach can be successfully applied (such as~\cite{SinclairSrivastavaStefankovicYin2017SpatialMixingandtheConnectiveConstantOptimalBounds}).

\begin{thm}\label{thm:VSSM at a distance and zeros}
    For each increasing unbounded function $\varphi \colon \mathbb N \to [0,\infty)$, each 
    $\Delta\in \mathbb{N}_{\geq 3}$ and $\lambda^\star>\lambda_c(\Delta)$, there exists a family of graphs $\mathcal{G}$ of maximum degree at most $\Delta$ such that $\mathcal{G}$ satisfies VSSM at distance $\varphi$ for every $\lambda\in [0,\lambda^\star]$ and such that there exists a sequence of graphs $(G_n)_n$ in $\mathcal{G}$ and a sequence $(\lambda_n)_n$ in $\mathbb{C}$ such that $Z_{G_n}(\lambda_n)=0$ and $\lambda_n\to \lambda_c(\Delta)$.
\end{thm}

\subsection{VSSM implies spectral independence}\label{ssec:spectral ind}
Spectral independence is an important property in the analysis of the mixing time of the Glauber dynamics for the hard-core model on bounded degree graphs. 
Anari, Liu and Oveis Gharan~\cite{spectralindependence-hardcore} showed in a breakthrough result that spectral independence implies rapid mixing of the Glauber dynamics and that the hard-core model on the family of all graphs of maximum degree at most $\Delta$ is spectrally independent at all fixed $\lambda<\lambda_c(\Delta)$, thereby matching applicability of Markov chain based randomized algorithms with that of the deterministic algorithms based on the approaches of Weitz~\cite{Weitz2006CountingIndependentSetsuptotheTreeThreshold} and Barvinok~\cite{Barvinok2016CombinatoricsandComplexityofPartitionFunctions}.
We refer to~\cite{spectralindsurvey} for a survey covering recent developments and improvements in this rapidly evolving direction of research. 
Let us point out one recent interesting development. In~\cite{LiuetalMarkovchainsphasetrasitions} it was shown that a certain proof approach for showing rapid mixing of the Glauber dynamics of the hard-core model at a parameter can be used to show zero-freeness in a neighbourhood of that parameter, thereby further linking zero-freeness to other notions in the field.

For completeness we will formally define the notion of spectral independence following~\cite{spectralindsurvey}.
Let $\lambda>0$. 
For a graph $G=(V,E)$ and distinct vertices $u,v\in V$ we define the \emph{influence} of $u$ on $v$ as 
\[
\Psi_{G,\lambda}(u\rightarrow v)=\mathbb{P}[v\in {\mathbf I}| u\in {\mathbf I}]-\mathbb{P}[v\in {\mathbf I}| u\notin{\mathbf I}].
\]
Define the $V\times V$ matrix $I_{G,\lambda}$ by $I_{G,\lambda}(u,v)=\Psi_{G,\lambda}(u\rightarrow v)$ if $u\neq v$ and $I_{G,\lambda}(u,v)=0$ otherwise.  

The hard-core model at fugacity $\lambda$ is called \emph{spectrally independent} on a family of graphs $\mathcal{G}$ if there exists a constant $\eta>0$ such that for each $G\in \mathcal{G}$ and each induced subgraph $H$ of $G$ the largest eigenvalue\footnote{It can be shown that all eigenvalues of $I_{G,\lambda}$ are nonnegative, see~\cite[Section 2.1]{spectralindsurvey}.} of $I_{H,\lambda}$ is bounded by $\eta$. (Note that our definition slightly differs from~\cite{spectralindsurvey} since we work with induced subgraphs rather than pinnings/boundary conditions.)

Anari, Liu and Oveis Gharan~\cite{spectralindependence-hardcore} showed that, for the family of all graphs of maximum degree $\Delta$, the hard-core model satisfies spectral independence at all fugacities $\lambda\in[0,\lambda_c(\Delta))$ building on Weitz' proof~\cite{Weitz2006CountingIndependentSetsuptotheTreeThreshold} of strong spatial mixing for hard-core model on the infinite tree. See also~\cite{Kuikuirapidmixing} for other usages of proof techniques for strong spatial mixing leading to proofs of spectral independence for other models.
It is therefore natural to wonder whether notions such as strong spatial mixing itself implies spectral independence. 
This question was explicitly asked by \v{S}tefankovi\v{c} in a lecture series at TU Dortmund in 2022\footnote{See \texttt{https://eac.cs.tu-dortmund.de/storages/eac-cs/r/summerschool-2022/Stefankovic-II.pdf}}.
Theorem~\ref{thm:main} in combination with~\cite[Theorem 11]{Spectralindependenceviastability} (which essentially says that multivariate zero-freeness implies spectral independence) directly implies the following corollary, thereby giving a positive answer to this question.
\begin{cor}\label{cor:vssmimpliesSI}
    Let $\Delta \in \mathbb{N}$ and let $\lambda^\star>0$.
    Let $\mathcal{G}$ be a family of graphs of maximum degree at most $\Delta$. If the hard-core model on $\mathcal{G}$ satisfies very strong spatial mixing at $\lambda^\star$, then the hard-core distribution at $\lambda^\star$ on $\mathcal{G}$ is spectrally independent.   
\end{cor}

\subsection{Proof approach}\label{ssec:proof overview}
As is well known, the ratio allows to transfer information about the partition function being zero to the ratio taking the value $-1$.

More concretely, let $G$ be a graph with vertex $v_0$.
Then if $Z_{G-v_{0}}(\lambda)\neq 0$, then
\begin{equation}\label{eq:ratio -1 vs Z=0}
    Z_{G}(\lambda)=0  \text{ if and only if }R_{G,v_{0}}(\lambda)=-1.
\end{equation}
(Note that the implication $R_{G,v_{0}}(\lambda)=-1$ implies $Z_{G}(\lambda)=0$ also holds without the assumption that $Z_{G-v_{0}}(\lambda)\neq 0$.)

To prove Theorem~\ref{thm:main} we make use of this simple observation in combination with the recursive nature of (rooted) trees by writing the ratios as a composition 
of maps encoding the contribution of fixed depth subtrees.
We aim to show that these tree ratio maps are contracting for small perturbations of $\lambda$ and use this to inductively show that the ratios stay close to the positive real axis and hence can never reach $-1$. This takes inspiration from~\cite{LSSCorrelationdecayandcomplexzeros}.

To show that the tree ratio maps are contracting, we write them as compositions of maps of the form $z\mapsto f_\lambda(z):=\frac{\lambda}{1+z}$ for varying values of $\lambda$. 
This takes inspiration from~\cite{PetersRegts2019OnaConjectureofSokalConcerningRootsoftheIndependencePolynomial,deBoeretalchaoticratios}. 
Here we think of these varying values of $\lambda$ as functions of a boundary condition. 
We use VSSM to see that the associated sequences of $\lambda$ for the two extremal boundary conditions (all leaves are either `in' or `out') tend to get close exponentially fast.
We view the sequence of compositions of the maps $f_\lambda$ as a non-autonomous dynamical system and use a carefully chosen non-autonomous change of coordinates to transform the M\"obius transformations $f_\lambda$ to affine maps.
We then use the assumption of VSSM once more to show that the tree ratio maps map the right half plane into a small disk.
This is then combined with the change of coordinates to show that the tree ratio maps are indeed contracting.

To prove Theorem~\ref{thm:VSSM at a distance and zeros} we construct a sequence of trees for which we show by hand that they satisfy $\varphi$-VSSM. To prove the statement about the zeros, we use the dynamical properties of the rational function $z\mapsto \frac{\lambda}{(1+z)^{\Delta-1}}$ as proved in~\cite{PetersRegts2019OnaConjectureofSokalConcerningRootsoftheIndependencePolynomial} in combination with Montel's theorem along similar lines as was done in~\cite{Buys2021CayleyTreesDoNotDeterminetheMaximalZeroFreeLocusoftheIndependencePolynomial,deBoeretalchaoticratios}.

In the next section, we collect some basic properties of the recursive computation of ratios of trees and show how to write the tree ratio map as a composition of the maps $f_\lambda$.
In Section~\ref{sec:dynamics} we study these compositions in more detail, describe the change of coordinates, and prove several consequences of this perspective under VSSM-like assumptions.
In Section~\ref{sec:proof of main} we collect all the ingredients to prove Theorem~\ref{thm:main} along the lines explained above.
Section~\ref{sec:vssm at distance} contains a proof of Theorem~\ref{thm:VSSM at a distance and zeros}.
In Section~\ref{sec:conclusion} we conclude with some questions that arise naturally and some final remarks.

\subsection*{Acknowledgements}This publication is part of the project \emph{Phase transitions, computational complexity and chaotic dynamical systems} with file number OCENW.M.22.155 of the research programme \emph{Open Competitie ENW M22-2} which is (partly) financed by the Dutch Research Council (NWO) under the grant \url{https://www.nwo.nl/en/projects/ocenwm22155}. Guus Regts is partially funded by the Netherlands Organisation of Scientific Research (NWO): VI.Vidi.193.068.

\section{Preliminaries on occupation ratios on trees}\label{sec:preliminaries}
We collect here some preliminaries that will be used throughout. 
Most of the results that we state here are well known, but we provide proofs for completeness and occasionally to be able to expand on these proofs.

We will start with a slight extension of the result of Weitz~\cite{Weitz2006CountingIndependentSetsuptotheTreeThreshold} we recalled in Lemma~\ref{lem:weitz}.
\begin{lem}\label{lem:ratio TSAW is ratio G}
    Let $G=(V,E)$ be a graph, let $v$ be a vertex of $G$ and let $(\lambda_u)_{u\in V}$ be a vector of variables. 
    Then there exists a vector of variables $(\tilde \lambda_w)_{w\in V(T_{\mathrm{SAW}(G,v)})}$ with $\tilde \lambda_w\in \{\lambda_u\mid u\in V\}$ for each vertex $w$ of $T_{\mathrm{SAW}(G,v)}$ such that
    \begin{equation}
        R_{T_{\mathrm{SAW}(G,v)},v}(\tilde\lambda_w)=R_{G,v}(\lambda_u).\label{eq:ratio TSAW is ratio G}  
    \end{equation}
    Moreover, for any vertex $u$ of $T_{\mathrm{SAW}(G,v)}$, the induced subtree $(T,u)$ of $T_{\mathrm{SAW}(G,v)}$, with those vertices $w$ such that $u$ is on the unique path from $w$ to $v$, is the self avoiding walk tree of an induced subgraph of $G$.
\end{lem}

\begin{proof}
    We will prove~\eqref{eq:ratio TSAW is ratio G} as well as the claim about the subtrees of $T_{\mathrm{SAW}(G,v)}$ by induction on the number of vertices of $G$.
    Fix an ordering of the vertices of $V$.
    We note that if the component of $G$ containing $v$ consists of one vertex, then $T_{\mathrm{SAW}(G,v)}$ is just a single vertex, we set $\tilde{\lambda}_v=\lambda_v$ and both statements clearly hold.
    
    Next, assume that $v$ has at least one neighbour and let $v_1,\ldots,v_d$ be its neighbours.
    Recall from Definition~\ref{def:TSAW} that for $i=1\ldots d$, we denote by $G_i$ the induced subgraph obtained by deleting the vertices $v,v_1,\ldots,v_{i-1}$ from $G$. 
    Denote for $i=1,\ldots, d$, $T_i=T_{\mathrm{SAW}(G_{i},v_{i}),v_i}$.
    Let $\tilde{\lambda}^i_u$ be the associated vectors of variables.
    By definition $T_{\mathrm{SAW}(G,v)}$ is the tree obtained from the disjoint union of the $T_i$ by adding a vertex $v$ and connecting it to each of the $v_1,\ldots,v_d$. 
    By induction we have that for each $i$, that there exists $ (\tilde{\lambda}^i_w)_{w\in V(T_i)}$ with  $\tilde{\lambda}^i_w\in \{\lambda_u\mid u\in V-v\}$ for each $i$ such that
    \begin{equation}\label{eq:ratio claim induction}
        R_{G_i,v_i}(\lambda_u)=R_{T_i,v_i}(\tilde \lambda^i_w)
    \end{equation}
    
    Now let let $\tilde\lambda_w$ be defined by $\tilde\lambda_v=\lambda_v$ and by letting $\tilde\lambda_w=\tilde \lambda^i_w$ in case $w$ appears in $T_i$.
    We claim that this implies~\eqref{eq:ratio TSAW is ratio G}.
    Indeed for the left-hand side of~\eqref{eq:ratio claim induction} we have, by a telescoping argument, using~\eqref{eq:fundamental recurrence} for each of the $G_i$
    \begin{align*}
        R_{G,v}=\lambda_v \frac{Z_{G\setminus N[v]}(\lambda_u}{Z_{G-v}(\lambda_u)}=\lambda_v \frac{Z_{G_d-v_d}(\lambda_u)}{Z_{G_1}(\lambda_u)}=\lambda_v \prod_{i=1}^d \frac{Z_{G_i-v_i}(\lambda_u)}{Z_{G_i}(\lambda_u)}=\frac{\lambda_v}{ \prod_{i=1}^d (1+R_{G_i,v_i}(\lambda_u))}
    \end{align*}
    While for the right-hand side of~\eqref{eq:ratio claim induction} we have, using that the independence polynomial is multiplicative over the components of $T_{\mathrm{SAW}(G,v)}-v$, 
    \begin{align*}
        R_{T_{\mathrm{SAW}(G,v)},v}(\tilde \lambda_w)=\frac{Z^{v\textnormal{ in}}_{T_{\mathrm{SAW}(G,v)}}(\tilde \lambda_w)}{Z^{v \textnormal{ out}}_{T_{\mathrm{SAW}(G,v)}}(\tilde \lambda_w)}=\tilde \lambda_v \prod_{i=1}^d \frac{Z_{T_i-v_i}(\tilde \lambda^i_w)}{Z_{T_i}(\tilde \lambda^i_w)}=\frac{\lambda_v}{\prod_{i=1}^d (1+R_{T_i,v_i}(\tilde \lambda_w))},
    \end{align*}
    as desired.
    
    Now for the second part of the statement, in case the vertex $u$ has distance $1$ to $v$ in $T_{\mathrm{SAW}(G,v)}$, then the tree $(T,u)$ is equal to one of the $T_i$ and we are done by construction.
    In case the distance of $u$ to $v$ is larger than $1$, $u$ is contained in one of the $T_i$ and the claim follows by induction as the tree $(T,u)$ is the self avoiding walk tree of an induced subgraph of $G_i$, and hence of an induced subgraph of $G$.
    This finishes the proof.
\end{proof}

Our strategy to prove Theorem~\ref{thm:main} is to show that the ratios avoid $-1$. Lemma~\ref{lem:ratio TSAW is ratio G} above indicates that it suffices to show this for ratios of trees. 
We therefore introduce some terminology related to ratios of trees.

Consider a rooted tree $(T,v)$ and $\lambda>0$.
As seen in the proof of Lemma~\ref{lem:ratio TSAW is ratio G}, the ratio 
\[
R_{T,v}=R_{T,v}(\lambda)=\frac{Z^{v \textnormal{ in}}_{T}(\lambda)}{Z^{v \textnormal{ out}}_{T}(\lambda)} 
\]
can be recursively computed as follows. Let $v_1,\ldots,v_d$ be the neighbours of $v$ and let $T_1,\ldots,T_d$ be the components of $T-v$ containing $v_1,\ldots,v_d$ respectively.
Then 
\begin{equation} \label{eq:recursion ratio tree}
    R_{T,v}=\frac{\lambda}{\prod_{i=1}^d (1+R_{T_i,v_i})}.
\end{equation}

For $n\in \mathbb{N}$ recall that $S_T(v,n)$ denotes the collection of vertices at distance $n$ from $v$ in $T$ and let $T_n$ be the subtree of $T$ induced by the vertices of distance at most $n$ from $v$.
For each $u\in S_T(v,n)$ let $R_u$ denote the ratio at $u$ in the component of $T\setminus S_T(v,n-1)$ containing $u$.
By repeating \eqref{eq:recursion ratio tree}, we can express the ratio $R_{T,v}$ as a rational function 
\begin{equation}\label{eq:define F_n}
    F_{T,n}:\mathbb{C}^{S_T(v,n)}\to \mathbb{C}
\end{equation}
applied to the vector $(R_u)_{u\in S_T(v,n)}$.
Next, denote by $S^*_T(v,n)$ the subset of vertices in $S_T(v,n)$ that are not leaves.
We will consider the restriction of $F_{T,n}$ to $\mathbb{C}^{S^*_T(v,n)}$, noting that for each $u\in S_T(v,n)\setminus S^*_T(v,n)$ we have $R_u=\lambda$, which we consider fixed.

Let $f_{\lambda}$ denote the M\"obius transformation defined by 
\[
z\mapsto \frac{\lambda}{1+z}.
\]
We wish to express the function $F_{T,n}$ as a composition of the  M\"obius transformations $f_\lambda$ with varying values of $\lambda$ when we fix all but one of the inputs of $F_{T,n}.$
To this end fix a vertex $u\in S^*(v,n)$ and fix values $(r_{u'})_{u'\neq u}$, $r_{u'}\in [0,\lambda]$.
Consider the function $F_{T,n;u}$ defined by $z\mapsto F_{T,n}(z,(r_{u'}))$.
For $i=1,\ldots,n$ let $w_i$ be the vertex at distance $i$ from $u$ on the unique path $P$ from $u$ to $v=w_n$.
Let ${T}_{n,i}$ be the component of $T_n\setminus (P-w_i)$ containing $w_i$ (see Figure~\ref{fig:TreeComps}).

\begin{figure}
    \includegraphics{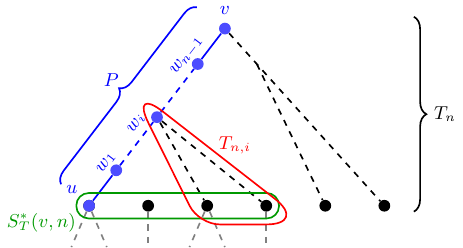}
    \caption{The components of the tree $T$.}
    \label{fig:TreeComps}
    
\end{figure}

For $\Delta\in \mathbb{N}$ and $\lambda>0$ we denote
\begin{equation}
    \ell_\Delta(\lambda):=\frac{\lambda}{(1+\lambda)^{\Delta-1}}. \label{eq:minimum value of lambda}
\end{equation}

\begin{lem}\label{lem:decomposing F into a composition of mobius transformations}
    Let $(T,v)$ be a rooted tree of maximum degree at most $\Delta\in \mathbb{N}$. 
    Let $n\in \mathbb{N}$ and let $u\in S^*_T(v,n)$.
    The function $F_{T,n;u}$ as defined above can be written as $f_{\lambda_n}\circ  \cdots f_{\lambda_1}$,
    where $\lambda_{i}$ is equal to the ratio $R_{T_{n,i},w_{i}}$ of the rooted tree $({T}_{n,i},w_i)$ for $i=1,\ldots, n$ and is contained in $[\ell_\Delta(\lambda),\lambda]$.
\end{lem}
\begin{proof}
    We prove this by induction on $n$. 
    In case $n=1$, the statement follows directly from~\eqref{eq:recursion ratio tree}.
    Indeed, denoting the neighbours of $v$ by $v_1,\ldots,v_{d}$ we may assume $u=v_1$ in which case we have 
    \begin{align*}
        F_{T,1}(z,r_{v_2},\ldots,r_{v_d})&=\frac{\lambda}{(1+z)\prod_{i=2}^{d} (1+r_{v_i})}
        \\
        &=\frac{\frac{\lambda}{\prod_{i=2}^d (1+r_{v_i})}}{1+z}=\frac{\lambda_1}{1+z}=f_{\lambda_1}(z),
    \end{align*}
    with, $\lambda_1:=\frac{\lambda}{\prod_{i=2}^d(1+r_{v_i})}=R_{T_{1,1},w_{1}}$. Clearly $\lambda_1\in [\frac{\lambda}{(1+\lambda)^{d-1}},\lambda]\subseteq [\ell_\Delta(\lambda),\lambda]$ proving the base case.
    
    Now let $n\geq 2$. 
    Let $v_1,\ldots,v_{d}$ be the neighbours of $v$. 
    As above, denote by $T_i$ the component of $T-v$ containing $v_i$.
    Assume that $v_1$ is on the unique path from $v$ to $u$.
    Then by~\eqref{eq:recursion ratio tree} we have 
    \[
    F_{T,n}(z,(r_{u'}))=\frac{1}{1+F_{T_1,n-1}(z,(r^1_{u'}))}\cdot \frac{\lambda}{\prod_{i=2}^{d}(1+F_{T_i,n-1}(r^i_{u'}))} 
    =R_{T_{n,n},w_{n}} ,
    \]
    where we denote by $(r^i_{u'})$ the vector of those $r_{u'}$ for which $u'$ are in the tree $T_i$.
    By induction we have 
    \[
    F_{T_1,n-1;u}(z,(r^1_{u'}))=f_{\lambda_{n-1}}\circ \cdots \circ f_{\lambda_1}(z),
    \]
    with $\lambda_i\in [\ell_\Delta(\lambda),\lambda]$.
    Setting
    \[
    \lambda_n:=\frac{\lambda}{\prod_{i=2}^{d}(1+F_{T_i,n-1}(r^i_{u'}))},
    \]
    we see that $F_{T,n}(z,(r_{u'}))=f_{\lambda_n}\circ f_{\lambda_{n-1}}\circ \cdots \circ f_{\lambda_1}(z)$, as desired.
    
    Since clearly $\lambda_n\leq \lambda$ it remains to prove that $\lambda_n\geq \ell_\Delta(\lambda)$.
    Again by induction we have that for each $i=2,\ldots,d$ fixing some $u^i\in S_{T_i}(v_i,n-1)$ that there exist $\lambda^i_1,\ldots,\lambda^i_{n-1}\in [\ell_\Delta(\lambda),\lambda]$ such that $F_{T_i,n-1}((r^i_{u'}))=f_{\lambda^i_{n-1}}\circ \cdots\circ f_{\lambda^i_1}(r^i_{u^i})$.
    In particular $F_{T_i,n-1}((r^i_{u'}))\leq \lambda$, implying that $\lambda_n\geq \frac{\lambda}{(1+\lambda)^{\Delta-1}}$, completing the proof.
\end{proof}

We next utilize part of this lemma to get an upper bound on the ratios for trees.
For $\lambda>0$ and $d\in \mathbb{N}$ denote 
\[
r_\Delta(\lambda):=\frac{\lambda}{1+\ell_\Delta(\lambda)}=f_\lambda(\ell_\Delta(\lambda)).
\]
\begin{lem}\label{lem:upper bound ratio}
    Let $\lambda>0$ and $\Delta\in \mathbb{N}$. Then for any rooted tree $(T,v)$ of maximum degree at most $\Delta$, consisting of at least two vertices, we have
    \[
    R_{T,v}(\lambda)\leq r_\Delta(\lambda).
    \]
\end{lem}
\begin{proof}
    Let $v_1,\ldots v_{d}$ be the neighbours of $v$ and let $T_i$ be the components of $T-v$ containing the $v_i$.
    Note that $d\geq 1.$
    Then we have 
    \[
    R_{T,v}(\lambda):=\frac{\lambda}{\prod_{i=1}^{d}(1+R_{T_i,v_i}(\lambda))}\leq \frac{\lambda}{1+R_{T_1,v_1}(\lambda)}\leq \frac{\lambda}{1+\ell_\Delta(\lambda)},
    \]
    where the last inequality is due to Lemma~\ref{lem:decomposing F into a composition of mobius transformations} above.
\end{proof}
In the next section we will study maps of the form $ F(z)=f_{\lambda_n}\circ \cdots \circ f_{\lambda_1}(z)$ developing tools to control the behaviour of the map $F_{T,n}$ which will be crucial for our proof of Theorem~\ref{thm:main}.

\section{Non-autonomous affine dynamics}\label{sec:dynamics}
Let $\lambda>0$ and let $(\lambda_n)$ be a sequence of numbers in $(0,\lambda]$.
Motivated by Lemma~\ref{lem:decomposing F into a composition of mobius transformations} we study the non-autonomous dynamical system 
\begin{equation}\label{eq:F to mobius}
    z\mapsto F(z)=f_{\lambda_n}\circ \cdots \circ f_{\lambda_1}(z).
\end{equation}
in this section. We use throughout elementary properties of M\"obius transformations, holomorphic functions, and the Poincar\'e distance that can be found in any book on geometric complex analysis, such as \cite{Zakeri2021ACourseinComplexAnalysis}.

For a starting value $w_0$ we write $(w_n)$ for the orbit of $w_0$ under the associated sequence of maps $f_{\lambda_i}$. 
In other words, $w_i=f_{\lambda_n}(w_{i-1})$ for $i\geq 1.$

\subsection{A non-autonomous conjugation to affine maps}




For a symbol $*\in \{<,>,\leq, \geq\}$ and $c\in \mathbb{R}$ we define the half plane
\[
\mathbb{H}_{*c} := \{z \in \widehat{\mathbb C} \mid  \mathrm{Re}(z) * c \},
\]
where by convention the point $\infty$ is included if and only if the inequality is not strict.

\begin{lem}\label{lem:construction of w_n}
    Let $\lambda>0$.
    Given a sequence $(\lambda_n)$ in $(0,\lambda]$ there exists a unique starting value $w_0$ for which the orbit $(w_n)$ is contained in the left half-plane $\mathbb{H}_{<0}$.
    Moreover, the orbit $(w_n)$ is contained in $(-\lambda-1,-1)$.
\end{lem}
\begin{proof}
    Throughout we write $f_n=f_{\lambda_n}$.
    Observe that the inverse map $f_n^{-1}$ is given by $f_n^{-1}(z)=-1+\lambda_n/z$ and maps the left half-plane into the shifted left half-plane $\mathbb H_{<-1}$.
    The Schwarz-Pick Lemma implies that each inverse map $f_n^{-1}$ defines a strict contraction with respect to the Poincar\'e metric on $\mathbb{H}_{<0}$. 
    On compact subsets of $\mathbb{H}_{<0}$ these contractions are uniform over all $\lambda_n \in (0,\lambda]$, since each $f_n^{-1}$ maps the left half-plane into $\mathbb{H}_{<-1}$.
    
    Consider the nested sequence of non-empty, connected, compact subsets of the extended complex plane $\widehat{\mathbb C}\coloneqq \mathbb{C} \cup \{\infty\}$:
    $$
    \mathbb{H}_{\leq 0} \supset f_1^{-1}(\mathbb{H}_{\leq 0}) \supset (f_2 \circ f_1)^{-1}(\mathbb{H}_{\leq 0}) \supset \cdots
    $$
    Then the intersection
    \begin{equation}\label{eq:intersection}
        \bigcap_{n\in \mathbb N} (f_n \circ \cdots \circ f_1)^{-1}(\mathbb{H}_{\leq 0})
    \end{equation}
    is a non-empty, connected, compact subset of $\mathbb{H}_{\leq 0}$. 
    Since all second preimages $f_{n-1}^{-1}\circ f_{n}^{-1} (\mathbb{H}_{\leq 0}),n\in \mathbb N$ are contained in a compact subset of $\mathbb{H}_{<0}$ and the contraction of the maps $f_n^{-1}$ is uniform over $n$, the diameter of the finite nested intersections must converge to $0$ and the intersection consists of a single point $w_0$, which is the only point whose entire forward orbit remains in $\mathbb{H}_{<0}$.
    
    Since the M\"obius transformations $f_n^{-1}$ have real parameters, it follows that the intersection~\eqref{eq:intersection} is invariant under conjugation and therefore each $w_n$ is real (and hence negative).
    This implies that $w_n< -1$ for each $n$. 
    Indeed, we have 
    \[
    w_{n}=f_{n+1}^{-1}(w_{n+1})=-1+\lambda_{n+1}/w_{n+1}< -1
    \] 
    using that $w_{n+1}<0$ for all $n$.
    This in turn implies that $w_n\geq -1-\lambda$ for each $n$. 
    Indeed 
    \[
    w_{n}=f_{n+1}^{-1}(w_{n+1})=-1+\lambda_{n+1}/w_{n+1}> -1-\lambda_{n+1}\geq -1-\lambda,
    \]
    using that $1/w_{n+1}> -1$ for all $n$.
    This finishes the proof.
\end{proof}

For a sequence ${\boldsymbol \lambda}=(\lambda_n)$ in $(0,\lambda]$ we write $f_n=f_{\lambda_n}$ for each $n$.
Let $(w_n)=(w_n(\boldsymbol \lambda))$ be the associated orbit that stays in $\mathbb{H}_{<0}$ 
We note that the point $w_n = f_n \circ \cdots \circ f_1(w_0)$ is the unique point whose orbit remains in $\mathbb{H}_{<0}$ for the maps $f_m \circ \cdots \circ f_{n+1}$, with $m > n$.

We next study the effect of exponentially smaller consecutive perturbations of the parameters $\lambda_n$ on the sequence $(w_n)$. Let us write $(\widehat{\lambda}_n)$ (and correspondingly $(\widehat{w}_n)$) for a second sequence of parameters.

\begin{lem}\label{lem:lambda close then also w close}
    Given $\lambda >0$, $C>0$ and $\alpha\in(0,1)$ there exists $C'>0$ such that the following holds. Given sequences $(\lambda_n)$ and $(\widehat\lambda_n)$ in $(0,\lambda]$ satisfying $|\lambda_n-\widehat\lambda_n|<C\alpha^n$ for all $n\in \mathbb N$, the associated orbits $(w_n)$  and $(\widehat w_n)$ satisfy
    \begin{equation*}
        |w_n-\widehat w_n|\leq C' \alpha^n\quad\text{for all }n\in \mathbb N.
    \end{equation*}
\end{lem}
\begin{proof}
    The sequence $(w_{n})_{n}$ was chosen such that
    \[
    \{w_{0}\}=\bigcap_{m\in\mathbb{N}}(f_{m}\circ\cdots\circ f_{1})^{-1}(\mathbb{H}_{<0}).
    \]
    In particular, for the point $-1\in\mathbb{H}_{<0}$, we have 
    \[
    w_{0}=\lim_{m\to\infty}(f_{m}\circ\cdots\circ f_{1})^{-1}(-1)
    \]
    and similarly 
    \begin{align*}
        w_{j} & =\lim_{m\to\infty}(f_{m}\circ\cdots\circ f_{j+1})^{-1}(-1)=\lim_{m\to\infty}w_{j,m},
    \end{align*}
    where 
    \[
    w_{j,m}:=f_{j+1}^{-1}\circ\cdots\circ f_{m}^{-1}(-1)=f_{j+1}^{-1}(w_{j+1,m}),
    \]
    for $m\ge j$. The analogous statements hold for $\widehat{w}_{j}$.
    We will show that there
    exists a constant $C_{1}>0$ independent of $m$ such that 
    \begin{equation}
        d_{\mathbb{H}_{<0}}(w_{j,m},\widehat{w}_{j,m})<C_{1}\alpha^{j}\quad\text{for all }j\leq m.\label{eq:wsPoincareClose}
    \end{equation}
    where $d_{\mathbb{H}_{<0}}$ denotes the Poincaré distance on $\mathbb{H}_{<0}$.
    
    Before we prove~\eqref{eq:wsPoincareClose} we first use it to prove the lemma.
    Note that all $w_{j,m},\widehat{w}_{j,m}$ lie in the compact subset $[-\lambda-1,-1]$
    of $\mathbb{H}_{<0}$ as follows by an easy induction, (cf. the proof of Lemma~\ref{lem:construction of w_n}).
    Since $d_{\mathbb{H}_{<0}}$ is equivalent to the Euclidean
    distance on  $[-\lambda-1,-1]$, (\ref{eq:wsPoincareClose}) implies that there is a constant $C'>0$
    such that 
    \[
    |w_{j}-\widehat{w}_{j}|=\lim_{m\to\infty}|w_{j,m}-\widehat{w}_{j,m}|<C'\alpha^{j}\quad\text{for all }j\in\mathbb{N}
    \]
    proving the lemma.
    
    We now return to proving~\eqref{eq:wsPoincareClose}.
    Let $C_{1}=\frac{C_{2}C\alpha}{1-\alpha}$, where $C_{2}>0$ is a
    constant such that $d_{\mathbb{H}_{<0}}(\zeta,\widehat{\zeta})\le C_{2}|\zeta-\widehat{\zeta}|$
    for all $\zeta,\widehat{\zeta}$ in the compact subset $[-1-\lambda,-1]$
    of $\mathbb{H}_{<0}$. (This choice of $C_{1}$ will become clear later).
    We will show (\ref{eq:wsPoincareClose}) by backwards induction on $j=m,m-1,\ldots,0$, note that $w_{m,m}=\widehat{w}_{m,m}=-1$
    and for $j\le m$ we have:
    \[
    d_{\mathbb{H}_{<0}}(w_{j-1,m},\widehat{w}_{j-1,m})\le d_{\mathbb{H}_{<0}}(f_{j}^{-1}(w_{j,m}),f_{j}^{-1}(\widehat{w}_{j,m}))+d_{\mathbb{H}_{<0}}(f_{j}^{-1}(\widehat{w}_{j,m}),\widehat{f}_{j}^{-1}(\widehat{w}_{j,m})).
    \]
    By induction and contraction of the Poincaré distance under the holomorphic
    map $f_{j}^{-1}:\mathbb{H}_{<0}\to\mathbb{H}_{<0}$, we have 
    \[
    d_{\mathbb{H}_{<0}}(f_{j}^{-1}(w_{j,m}),f_{j}^{-1}(\widehat{w}_{j,m}))\le d_{\mathbb{H}_{<0}}(w_{j,m},\widehat{w}_{j,m})\le C_{1}\alpha^{j}
    \]
    and by our choice of $C_{2}$, we have 
    \begin{align*}
        d_{\mathbb{H}_{<0}}(f_{j}^{-1}(\widehat{w}_{j,m}),\widehat{f}_{j}^{-1}(\widehat{w}_{j,m})) 
        &= d_{\mathbb{H}_{<0}} \paren[{\Big}]{-1+\frac{\lambda_{j}}{\widehat{w}_{j,m}},-1+\frac{\widehat{\lambda}_{j}}{\widehat{w}_{j,m}}}\le C_2\abs[\Big]{\frac{\lambda_{j}}{\widehat{w}_{j,m}}-\frac{\widehat{\lambda}_{j}}{\widehat{w}_{j,m}}}
        \\
        &\leq C_{2}|\lambda_{j}-\widehat{\lambda}_{j}|<C_{2}C\alpha^{j},
    \end{align*}
    where the first inequality uses that $\widehat{w}_{j,m}\leq -1$ and the second inequality uses the assumption in the lemma.
    Combining these estimates, we obtain:
    \[
    d_{\mathbb{H}_{<0}}(w_{j-1,m},\widehat{w}_{j-1,m})\le(C_{1}\alpha+C_{2}C\alpha)\alpha^{j-1}=C_{1}\alpha^{j-1}
    \]
    by our choice of $C_{1}$. This completes the proof of (\ref{eq:wsPoincareClose})
    and hence the lemma. 
\end{proof}

We use the orbit $(w_n)$ to change coordinates non-autonomously via:
\[
\phi_n(z) = \phi_{w_n}(z) = \frac{1}{z-w_n},
\]
and obtain a sequence of affine maps:
\begin{equation}\label{eq:define g_n}
    g_n = \phi_n \circ f_n \circ \phi_{n-1}^{-1},
\end{equation}
see Figure~\ref{fig:diagram} for the associated commutative diagram.
\begin{figure}[h] 
    \[
    \begin{tikzcd}
        \widehat{\mathbb{C}} \arrow[r, "f_1"] \arrow[d, "\phi_0"'] 
        &  \widehat{\mathbb{C}}  \arrow[r, "f_2"] \arrow[d, "\phi_1"'] 
        & \widehat{\mathbb{C}}  \arrow[r, "f_3"]  \arrow[d, "\phi_2"]
        & \widehat{\mathbb{C}} \arrow[d, "\phi_3"'] \mathrlap{{}\cdots}
        \\
        \widehat{\mathbb{C}}  \arrow[r, "g_1"']              
        & \widehat{\mathbb{C}}  \arrow[r, "g_2"']              
        & \widehat{\mathbb{C}} \arrow[r, "g_2"']              
        & \widehat{\mathbb{C}}  \mathrlap{{}\cdots}
    \end{tikzcd}
    \]
    
    \caption{A commutative diagram displaying the relation between the $f_n$ and the $g_n$.\label{fig:diagram}}
\end{figure}
Since $f_n$ maps $w_n$ to $w_{n+1}$ and $\phi_n(w_n)=\infty$, the M\"obius transformation $g_n$ maps $\infty$ to $\infty$, and is therefore \emph{affine}. Let us write down the precise formula for $g_n$:
\[
\begin{aligned}
    g_n(\zeta) = {} & \phi_n \circ f_n \circ \phi_{n-1}^{-1}(\zeta)\\
    = {} & \phi_n \circ f_n( \frac{1}{\zeta} + w_{n-1})\\
    = {} & \phi_n \paren[\bigg]{\frac{\lambda_n \zeta}{1+\zeta w_{n-1} + \zeta}}\\
    = {} & \frac{1 + (1+w_{n-1})\zeta}{\lambda_n \zeta - w_{n}(1+ (1+w_{n-1})\zeta)}\\
    = {} & - \frac{1 + (w_{n-1} + 1) \zeta}{w_n},
\end{aligned}
\]
where we used the fact that $g_n$ is affine to obtain the last equality. 


Let $G_n=g_n\circ \cdots \circ g_1$.
Then
\begin{align}\label{eq:G_n prime}
    G_n^\prime(\zeta)=(g_n \circ \cdots \circ g_1)^\prime(\zeta) = (-1)^n \prod_{j=1}^n \frac{w_{j-1} + 1}{w_j} = (-1)^n \frac{w_0+1}{w_n}  \cdot \prod_{j=1}^{n-1} \frac{w_j + 1}{w_j}.
\end{align}

Following the notation from Lemma \ref{lem:lambda close then also w close} we write $\widehat{g}_n$ and $\widehat{G}_n$ for the maps induced by the second sequence $(\widehat{\lambda}_n)$.

\begin{cor}\label{cor:derivative bound on ratio of G_n}
    Given $\lambda_+>\lambda_->0$, $C>0$ and $\alpha\in (0,1)$ there exists a constant $A>1$ such that the following holds: if $(\lambda_n)$ and $(\widehat\lambda_n)$ in $[\lambda_-,\lambda_+]$  satisfy $|\lambda_n-\widehat \lambda_n|<C\alpha^n$ for all $n$, then 
    \begin{equation}
        \label{eq:ratioOfGDerivativesBound}
        1/A\leq \abs[\bigg]{\frac{G_{n}'}{\widehat{G}_{n}'}} \leq A.
    \end{equation}
\end{cor}
\begin{proof}
    By \eqref{eq:G_n prime}, we know for all $n$ that 
    \begin{equation}\label{eq:QuotientOfDerivatives}
        \abs[{\bigg}]{\frac{G_{n}'}{\widehat{G}_{n}'}} = \prod_{j=1}^{n}\frac{w_{j-1}+1}{\widehat{w}_{j-1}+1}\frac{\widehat{w}_{j}}{w_{j}}.
    \end{equation}
    On the compact set $[-\lambda_{+}-1,-\lambda_{-}]^{2}$ both $\frac{\widehat{w}}{w}$
    and $\frac{w+1}{\widehat{w}+1}$ have bounded gradient and are equal to
    $1$ on the diagonal, so there is a constant $C_{3}>0$ such that we have 
    \[
    \frac{\widehat{w}}{w},\frac{w+1}{\widehat{w}+1} < 1+C_{3}|w-\widehat{w}|
    \]
    for all $w,\widehat{w}\in[-\lambda_{+}-1,-\lambda_{-}]$. Since all $w_j,\widehat{w}_j$ lie in this interval, (\ref{eq:QuotientOfDerivatives}) then implies 
    \[
    \abs[{\bigg}]{\frac{G_{n}'}{\widehat{G}_{n}'}} \le \prod_{j=1}^{n}(1+C_{3}|w_j - \widehat{w}_j|)^{2} < \prod_{j=1}^{n}(1+C_{3}C'\alpha^{j})^{2}.
    \]
    with $C'$ from Lemma~\ref{lem:lambda close then also w close}. 
    This product is bounded by 
    \[A:= \paren*{e^{C_3C'\tfrac{\alpha}{1-\alpha}}}^2.
    \]
    This shows the second inequality in \eqref{eq:ratioOfGDerivativesBound} and the first inequality follows by symmetry. 
\end{proof}

Next, we relate the supremum for a compact set $K$,
\[
\norm{F_n'}_K \coloneqq \sup_{z\in K}\abs{F_n'(z)} 
\]
of the absolute value of the derivative of $F_n=f_n\circ \cdots \circ f_1$ to that of $G_n=g_n\circ \cdots \circ g_1$.

\begin{lem}
    \label{lem:compare Derivative-general}Let $K,L\subseteq\mathbb{C}$
    be disjoint compact subsets. Then there exists a constant
    $C>0$ (depending only on $K$ and $L$) such that for any differentiable
    map $F:K\to K$ and 
    \[
    G=\phi_{w'}\circ F\circ\phi_{w}^{-1}:\phi_{w}(K)\to\phi_{w'}(K)
    \]
    where $\phi_{w}(z)=\frac{1}{z-w}$ with $w,w'\in L$, we have 
    \[
    \frac{1}{C}\norm{F'}_K<\norm{G'}_{\phi_w(K)}<C\norm{F'}_K.
    \]
\end{lem}
\begin{proof}
    As the derivatives $\phi_{w}'(z)=-\frac{1}{(z-w)^{2}}$ are uniformly
    bounded away from $0$ and $\infty$ for $z\in K$ and $w\in L$,
    there is a uniform bound $\sqrt{C}>0$ on the derivatives
    of $\phi_{w}|_{K}$ and $\phi_{w}^{-1}|_{\phi_{w}(K)}$ for all $w\in L$.
    Hence, we have 
    \[
    \norm{G'}_{\phi_w(K)}
    \le \norm{\phi_{w'}'}_{K} \norm{F'}_K \norm{(\phi_{w}^{-1})'}_{\phi_{w}(K)}
    \le C \norm{F'}_K
    \]
    and vice versa.
\end{proof}


\begin{cor}\label{cor:compare Derivative}
    Let $\lambda>0$, let $K\subset \mathbb{H}_{>-1/2}$ be a compact set. Then there exists a constant $C>0$ such that the following holds. 
    
    For any sequence ${\boldsymbol \lambda}=(\lambda_n)$ in $(0,\lambda]$,
    let $f_n:=f_{\lambda_n}$, $F_n(z)=f_n\circ \cdots \circ f_1$ and let $(w_n)$ be the unique orbit that stays in $\mathbb{H}_{<0}$.
    Let $(g_n)$ be the associated sequence of affine transformations defined by~\eqref{eq:define g_n} and let $G_n(z):=g_n\circ \cdots \circ g_1$.
    Then, if $\norm{G_n'}_{\phi_0(K)} \le \eta$, then $\norm{F_n'}_{K} \le C\eta$ and conversely, if $\norm{F_n'}_{K} \le \eta$, then $\norm{G_n'}_{\phi_0(K)} \le C \eta$.
\end{cor}
\begin{proof}
    For any $\lambda'\in[0,\lambda]$, the Möbius transformation $f_{\lambda'}$
    maps lines and circles to lines and circles and preserves angles.
    Since $f_{\lambda'}$ maps the real line to itself and is decreasing
    there, $f_{\lambda'}(\mathbb{H}_{>-1/2})$ is the disc $B_{\lambda'}$
    in the interior of the circle centred on the real line through $f_{\lambda'}(\infty)=0$
    and $f_{\lambda'}(-1/2)=2\lambda'$.
    
    Let $K'\subseteq\mathbb{H}_{>-1/2}$ be a compact set (e.g. a
    large closed ball) containing both $K$ and the closed ball $\overline{B}_{\lambda}$.
    Then for any $\lambda'\in[0,\lambda]$, we have $f_{\lambda'}(\mathbb{H}_{>-1/2})=B_{\lambda'}\subseteq\overline{B}_{\lambda}\subseteq K'$
    and hence $F_{n}(K')\subseteq K'$ independently of the sequence $(\lambda_{n})$.
    As $G_{n}=\phi_{n}\circ F_{n}\circ\phi_{0}^{-1}$, taking $L=[-\lambda-1,-1]$, which is clearly disjoint from $K'$, the statement follows
    from Lemma~\ref{lem:compare Derivative-general}
\end{proof}

The next lemma gives us information about the derivative of $F_n=f_n\circ \cdots \circ f_1$.

\begin{lem}
    \label{lem:Derivative-via-image-diameter}
    Let $D$ be a disc centred
    on the real line and $K\subseteq D$ a compact subset. Let $c_{1}$
    and $c_{2}$ be the two points where the boundary of $D$ intersects
    the real line $\mathbb{R}$. Then there exists a constant $C>0$ such
    that for any real Möbius transformation $F$ with $F(D)\subseteq\mathbb{C}$,
    we have 
    \[
    |F(z)-F(w)|<C|F(c_{2})-F(c_{1})|\cdot|z-w|
    \]
    for all $z,w\in K$.
\end{lem}

\begin{proof}
    On the compact subset $K\subseteq D$, the Poincaré metric $d_{D}$
    on $D$ is comparable to the Euclidean metric. So there exists a constant
    $C>0$ such that for $z,w\in K$, we have 
    \[
    |z-w|>\frac{1}{C}d_{D}(z,w).
    \]
    As before, the image $F(D)$ under the real Möbius transformation
    $F$ is bounded by the circle centred on the real line through $F(c_{1})$
    and $F(c_{2})$. Since it is contained in $\mathbb{C}$ by assumption,
    $F(D)$ must be the interior of that circle. The Poincaré distance
    on a disc of radius $r$ is bounded from below by $1/r$ times the
    Euclidean distance. Hence, by non-expansivenes of the Poincaré distance
    under the holomorphic map $F$, we conclude: 
    \[
    C|z-w|>d_{D}(z,w)\ge d_{F(D)}(F(z),F(w))\ge\frac{2}{|F(c_{2})-F(c_{1})|}|F(z)-F(w)|,
    \]
    where $d_{F(D)}$ denotes the Poincaré distance on $F(D)$. 
\end{proof}

The following corollary will be used in the proof of our main theorem.

\begin{cor}\label{cor:derivative bound F}
    Let $\Delta\in\mathbb{N}$ and let
    $\lambda>0$. Let $K\subset B_{\lambda/2}(\lambda/2)$ be a compact
    set containing a neighbourhood of $[\ell_{\Delta}(\lambda),r_{\Delta}(\lambda)]$.
    There exists a constant $C>0$ such that for any sequence $(\lambda_{n})$
    in $[\ell_{\Delta}(\lambda),\lambda]$, for all $n\in\mathbb{N}$
    and the maps $F_{n}=f_{\lambda_{n}}\circ\cdots\circ f_{\lambda_{1}}$,
    we have 
    \begin{equation}
    \norm{F_{n}'}_{K}\leq C|F_{n}(0)-F_{n}(\lambda)|.\label{eq:derivative bound F}
    \end{equation}
\end{cor}
\begin{proof}
    Possibly enlarging $K$, we may assume $K$ to be connected. Fix $n\in\mathbb{N}$. Clearly $F_{n}(B_{\lambda/2}(\lambda/2))\subset\mathbb{C}$.
    So Lemma~\ref{lem:Derivative-via-image-diameter} applied to $K$
    and $D=B_{\lambda/2}(\lambda/2)$ provides a constant $C>0$ such
    that for all $x,y\in K$ we have 
    \[
    |F_{n}(x)-F_{n}(y)|\leq C|F_{n}(0)-F_{n}(\lambda)||x-y|.
    \]
    Since $K$ is connected, this implies \eqref{eq:derivative bound F}.
\end{proof}

\section{Proof of Theorem~\ref{thm:main}} \label{sec:proof of main}

In this section we prove Theorem~\ref{thm:main}.
We use the framework developed in Section~\ref{sec:dynamics} above to provide a proof that VSSM implies absence of zeros. 

We start with a simple but crucial monotonicity property of the map $F_{T,n}$ as defined in~\eqref{eq:define F_n}.
\begin{lem}\label{lem:monotone behaviour of F}
    Let $(T,v)$ be a rooted tree.
    If $n$ is even the function $F_{T,n}$ is monotonically increasing in each coordinate, while if $n$ is odd the function $F_{T,n}$ is monotonically decreasing in each coordinate.
\end{lem}
\begin{proof}
    It suffices to prove this for $F_{T,n;u}$ for any $u\in S^*_T(v,n)$ and any fixed inputs $r_{u'}$ for the remaining $u'\in S^*_T(v,n)$.
    By Lemma~\ref{lem:decomposing F into a composition of mobius transformations} we have $F_{T,n;u}=f_{\lambda_n}\circ  \cdots f_{\lambda_1}$ for certain $\lambda_i>0$.
    Now since $f_{\lambda_i}^\prime(z)=\frac{-\lambda_i}{(1+z)^2}<0$ for all $z\geq 0$, and since $f_{\lambda_i}$ maps $\mathbb{R}_{\geq 0}$ into $\mathbb{R}_{\geq 0}$, it follows by the chain rule that 
    \[
    (-1)^n\frac{d}{dz} f_{\lambda_n}\circ \cdots \circ f_{\lambda_1}(z)>0.
    \]
    This finishes the proof.
\end{proof}

Let $(T,v)$ be a rooted tree. Let $n\in \mathbb{N}$ and recall that $S^*_{T}(v,n)$ denotes the collection of those vertices at distance $n$ from $v$ that are not leaves. 

The following is a direct consequence of the previous lemma.
\begin{cor}\label{cor:vssm =ssm for trees}
    Let $\lambda>0$.
    For any rooted tree $(T,v)$ and any $r=(r_u)\in [0,\lambda]^{S^*_T(v,n)}$ we have 
    \begin{eqnarray}
        {F}_{T,n}({\Vec{0}})\leq {F}_{T,n}(r_u) \leq {F}_{T,n}({\Vec{\lambda}}) & \text{ if $n$ is even, and} \label{eq:sandwhich even odd}
        \\
        {F}_{T,n}({\Vec{\lambda}})\leq {F}_{T,n}(r_u) \leq {F}_{T,n}({\Vec{0}})  &\text{ if $n$ is odd.}\label{eq:sandwhich even}
    \end{eqnarray}
\end{cor}


Fix a tree $T$ with vertex $v$ and let $n$ be an positive integer.
Identify $S^*_T(v,n)$ with $[\ell]:=\{1,\ldots,\ell\}$.
Let for $i=0,\ldots,\ell$ and $\lambda>0$ the vector $\delta^i\in \{0,\lambda\}^\ell$ be defined by $\delta^i_j=0$ if $j\leq i$ and $\delta^i_j=\lambda$ if $j>i$.
Let for $i=1,\ldots,\ell$ 
\begin{equation}
    \epsilon_i(T,n):={F}_{T,n}(\delta^i)-{F}_{T,n}(\delta^{i-1})
\end{equation}
and note that by Lemma~\ref{lem:monotone behaviour of F} $\epsilon_i(T,n)$ is positive for all $i=1,\ldots,\ell$ if $n$ is even and negative for all $i=1,\ldots,\ell$ if $n$ is odd.
Observe that by construction we have
\begin{equation}\label{eq:eq:decompose SSM using telescoping}
    R_{(T,\sigma_{n+1,0}),v}(\lambda)-R_{(T,\sigma_{n+1,1}),v,}(\lambda)={F}_{T,n}({\Vec{\lambda}})-{F}_{T,n}({\Vec{0}})=\sum_{i=1}^\ell \epsilon_i(T,n),
\end{equation}
where $\sigma_{n+1,0}$ (resp. $\sigma_{n+1,1}$) denotes the boundary condition on the vertices at distance $n+1$ from $v$ in $T$ assigning all vertices to $0$ (resp. $1$).

The following proposition tells us that the functions $F_{T,n;i}$ for $i\in S^*_T(v,n)$ defined before Lemma~\ref{lem:decomposing F into a composition of mobius transformations} contract with rate a constant times $|\epsilon_i(T,n)|$, provided $T$ is the self-avoiding walk tree of a graph $G$ contained in a family satisfying VSSM at $\lambda$.

\begin{prop}\label{prop:point to point contraction  wrp to sup norm}
    Let $\lambda>0$ and let $\Delta\in \mathbb{N}$.
    Let $\mathcal{G}$ be a family of graphs of maximum degree at most $\Delta$ that is closed under taking induced subgraphs and that satisfies VSSM with rate $\alpha\in [0,1)$ and constant $C'>0$ for the hard-core model at fugacity $\lambda$. 
    Let $\mathcal{T}$ denote the family of self avoiding walk trees of graphs in $\mathcal{G}$.
    Let $K\subset B_{\lambda/2}(\lambda/2)$ be a compact set containing a neighbourhood of $[\ell_{\Delta}(\lambda),r_{\Delta}(\lambda)]$.
    Then there exists a constant $C>0$ such that for any $(T,v)\in \mathcal{T}$, $n\in \mathbb{N}$, $u\in S^*_T(v,n)$, and any vector $(r_{u'})\in [\ell_\Delta(\lambda),\lambda]^{S^*_T(v,n)\setminus \{u\}}$ we have
    \[
    \norm{F'_{T,n;u}}_{K}\leq C|\epsilon_u(T,n)|.
    \]
\end{prop}

\begin{proof}
    Let us fix $(T,v)=(T_{\mathrm{SAW}(G,v)},v)\in \mathcal{T}$, an integer $n$, any $u\in S^*_T(v,n)$ and a vector $r\in [\ell_\Delta(\lambda),\lambda]^{S^*_T(v,n)\setminus \{u\}}$.
    After identifying $S^*_{T}(v,n)$ with $[\ell]$ we may assume that $u\in [\ell]$. 
    We will write $\epsilon=|\epsilon_u(T,n)|.$
    
    Let $\lambda_j$, for $j=1,\ldots,n$ be obtained from Lemma~\ref{lem:decomposing F into a composition of mobius transformations} from the input vector $\delta_u$ as defined right after Corollary~\ref{cor:vssm =ssm for trees}. Let $\tilde{\lambda}_j$ for $j=1,\ldots,n$ be obtained from the input vector $r$.
    By setting $\lambda_{n+i}=\tilde{\lambda}_{n+i}=\lambda$ for $i\geq 1$ we can of course turn these into infinite sequences.
    We write $f_j=f_{\lambda_j}$ and $\tilde{f}_{j}=f_{\tilde\lambda_j}$.
    Let $(w_j)$ (resp.) ($\tilde{w}_j)$ be the associated sequences obtained from Lemma~\ref{lem:construction of w_n}.
    Finally, let $(g_j)$ (resp. ($\tilde{g}_j$) be the corresponding sequences of affine maps.
    
    We first claim that for all $j\geq 1$ we have
    \begin{equation}
        |\lambda_j-\tilde{\lambda}_j|\leq C\alpha^j.     \label{eq:cliam exp close}
    \end{equation}
    
    It suffices to prove~\eqref{eq:cliam exp close} for $j\leq n$.
    By Lemma~\ref{lem:decomposing F into a composition of mobius transformations} we know that 
    \[
    \lambda_j={F}_{{{T}_{n,j},w_j}}(\delta_u),
    \]
    where $w_j$ is the vertex on the unique path $P$ from $v$ to $u$ in $T$ at distance $n-j$ from $v$ and where ${T}_{n,j}$ is the component of $T\setminus (P-w_j)$ containing the vertex $w_j$ (see Figure~\ref{fig:TreeComps}). 
    By Lemma~\ref{lem:monotone behaviour of F} it thus follows that $\lambda_j$ is sandwiched between the values ${F}_{{T}_{n,j},w_j}(\Vec{0})$ and ${F}_{{T}_{n,j},w_j}(\Vec{\lambda})$.
    Similarly we have 
    \[
    \tilde{\lambda}_j={F}_{{{T}_{n,j},w_j}}(r),
    \]
    and since $r_i\in [0,\lambda]$ for all $i\in [\ell]$, it follows from Lemma~\ref{lem:monotone behaviour of F} that also $\tilde{\lambda}_j$ is sandwiched between the values ${F}_{{T}_{n,j},w_j}(\Vec{0})$ and ${F}_{{T}_{n,j},w_j}(\Vec{\lambda})$.
    
    Fix $j$ and denote the neighbours of $w_j$ that are at distance $n-j+1$ from $v$, but are not on the path from $u$ to $v$ by $u_1,\ldots,u_{d'}$ and denote by $S_1,\ldots,S_{d'}$ the rooted trees containing $u_i$ and all vertices at distance larger than $n-j+1$ whose unique path to $v$ goes through $u_i.$
    Lemma~\ref{lem:ratio TSAW is ratio G} tells us that $(S_i,u_i)$ is the self avoiding walk tree of $(H_i,u_i)$ where $H_i$ is an induced subgraph of $G$.
    Since the family $\mathcal{G}$ satisfies VSSM at $\lambda$, there exists $C'>0$ and $\alpha\in (0,1)$ such that
    \[
    |R_{(S_i,\sigma_{j,0}),u_i}(\lambda)-R_{(S_i,\sigma_{j,1}),u_i}(\lambda)|\leq C'\alpha^{j}.
    \]
    Now by the construction of ${F}_{T,n}$ we have
    \begin{align*}
        {F}_{{T}_{n,j},w_j}(\Vec{0})=\frac{\lambda}{\prod_{i=1}^{d'} (1+R_{(S_i,\sigma_{j,1}),u_i}(\lambda))},
    \end{align*}
    and
    \begin{align*}
        {F}_{{T}_{n,j},w_j}(\Vec{\lambda})=\frac{\lambda}{\prod_{i=1}^{d'} (1+R_{(S_i,\sigma_{j,0}),u_i}(\lambda))},
    \end{align*}
    and thus there exists a constant $C>0$ only depending on $\lambda,C'$ and $d$ such that 
    \[
    |{F}_{{T}_{n,j},w_j}(\Vec{0})-{F}_{{T}_{n,j},w_j}(\Vec{\lambda})|\leq C\alpha^{j}.
    \]
    This proves~\eqref{eq:cliam exp close}.
    
    Let us next write $G_n=g_n\circ \cdots \circ g_1$, $\tilde{G}_n=\tilde{g}_n\circ \cdots \tilde{g}_1$ and $F_n=f_n\circ \cdots \circ f_1$ and $\tilde{F}_n=\tilde{f}_n\circ \cdots \tilde{f}_1$.
    Observe that with this notation we have
    $\tilde F_n(\,\cdot\,)=F_{T,n}(\,\cdot\,,r)$ and $F_n(\,\cdot\,)=F_{T,n}(\,\cdot\,,\delta_u)$.
    We will transfer knowledge about $\norm{F'_n}_{K}$ to $\norm{\tilde F'_n}_{K}$ using $G_n$ and $\tilde G_n$.
    
    Let $A>1$ be the constant from Corollary~\ref{cor:derivative bound on ratio of G_n} on input of $\lambda_+=\lambda$, $\lambda_-=\ell_\Delta(\lambda)$, $C$ and $\alpha$. Then we know that 
    \begin{equation}\label{eq:derivative bound on ratio of G_n}
        1/A\leq \abs[\bigg]{\frac{G_n^\prime}{\tilde{G}_n^\prime}} \leq A.
    \end{equation}

    Let $C^{(1)}$ be the constant from Corollary~\ref{cor:derivative bound F}.
    Then 
    \[
    \norm{F_n'}_{K}\leq C^{(1)}\epsilon.
    \]
    By Corollary~\ref{cor:compare Derivative} we thus have that 
    \[
    \abs{G'_n}=\norm{G'_n}_{\phi_0(K)}\leq C^{(2)}\epsilon.
    \]
    Consequently by~\eqref{eq:derivative bound on ratio of G_n} and
    another application of Corollary~\ref{cor:compare Derivative} we obtain 
    \[
    \norm{\tilde F_n'}_{K}\leq C^{(3)}\epsilon,
    \]
    where $C^{(3)}$ is equal to the constant from  Corollary~\ref{cor:compare Derivative} multiplied by $AC^{(2)}$.
    Since the constant $C^{(3)}$ only depends on $\lambda,C',\alpha$ and $\Delta$, this finishes the proof.
\end{proof}

We next use the previous proposition to show that small perturbations of the functions $F_{T,n}$ have no big effect on the value of the functions.
To this end let us denote for a rooted tree $(T,v)$ and $\Vec{\lambda}=(\lambda_u)_{u\in V(T)}$ the map $F_{T,n,\Vec{\lambda}}$ for the map obtained from $F_{T,n}$ by replacing in each step of the recursion~\eqref{eq:recursion ratio tree} at some root vertex $u$ the value of $\lambda$ by $\lambda_u$.

\begin{prop}\label{prop:contraction  wrp to sup norm}
    Let $\lambda>0$ and $\Delta\in\mathbb{N}$.
    Let $\mathcal{G}$ be a family of graphs of maximum degree at most $\Delta$ that is closed under taking induced subgraphs and satisfies VSSM with rate $\alpha\in [0,1)$ and constant $C'>0$ for the hard-core model at fugacity $\lambda$. 
    Let $\mathcal{T}$ denote the family of self avoiding walk trees of graphs in $\mathcal{G}$.
    Then there exists $n\in \mathbb{N}$ and $\eta_0>0$ such that for each $0<\eta_1<\eta_0$ there is an $\eta_2>0$ such that for any $(T,v)\in \mathcal{T}$, $r\in [\ell_\Delta(\lambda),r_\Delta(\lambda)]^{S^*_T(v,n)}$, $r'\in B_{\ell_\infty}(r,\eta_1)$ and  $\vec{\lambda}=(\lambda_u)_{u\in V(T)}$ with $\lambda_u\in B(\lambda,\eta_2)$ for each $u$ we have
    \begin{equation}\label{eq:goal}
        | {F}_{T,n,\Vec{\lambda}}(r') -{F}_{T,n}(r)|\leq \eta_1.
    \end{equation}
\end{prop}
\begin{proof}
    Choose $n$ large enough so that, with the constant $C$ from Proposition~\ref{prop:point to point contraction  wrp to sup norm} above, upon input of $\lambda$ and a compact set $K$ containing a neighbourhood of $[\ell_\Delta(\lambda),r_\Delta(\lambda)]$ and $\alpha,C'$ we have $CC'\alpha^n\leq 1/4$.
    Since there are only finitely many rooted trees of maximum degree at most $\Delta$ and depth $n+1$, it follows that there exists only finitely many associated functions $F_{T,n}$.
    
    Fix any rooted tree $(T,v)\in \mathcal{T}$.
    Proposition~\ref{prop:point to point contraction  wrp to sup norm} implies that for each $u\in S^*_T(v,n)$ and $r\in [\ell_\Delta(\lambda),\lambda]^{S^*_T(v,n)\setminus\{u\}}$ the absolute value of the derivative of $F_{T,n;u}$ is bounded by $C|\epsilon_u(T,n)|$ on a neighbourhood of $[\ell_\Delta(\lambda),r_\Delta(\lambda)]$.
    This implies that the $\ell_1$-norm of the gradient, $\|\nabla F_{T,n}(r)\|_1$, is bounded by 
    \begin{align*}
        \|\nabla F_{T,n}(r)\|_1 & =\sum_{u\in S^*_T(v,n)}\abs[\bigg]{\frac{\partial}{\partial r_u} F_{T,n}(r)}=\sum_{u\in S^*_T(v,n)}|F^\prime_{T,n;u}(r_u)|
        \\
        &\leq \sum_{u\in S^*_T(v,n)}C|\epsilon_u(T,n)|
    \end{align*}
    for all $r\in [\ell_\Delta(\lambda),r_\Delta(\lambda)]^{S^*_T(v,n)}$.
    Thus from \eqref{eq:eq:decompose SSM using telescoping}, the fact that all $\epsilon_u(T,n)$ have the same sign, the assumption of VSSM, and our choice of $n$, it follows that 
    \[
    \|\nabla F_{T,n}(r)\|_1\leq CC'\alpha^n\leq 1/4,
    \]
    for any $r\in [\ell_\Delta(\lambda),r_\Delta(\lambda)]^{S^*_T(v,n)}$ and all $(T,v)\in \mathcal{T}.$
    
    By continuity of the map $r\mapsto \nabla F_{T,n}(r)$ and the fact there exists only finitely many maps $F_{T,n}$, it follows that that there exists $\eta_1>0$ such that for all $r\in B([\ell_\Delta(\lambda),r_\Delta(\lambda)],\eta_1)^{S^*_T(v,n)}$ and all $(T,v)\in \mathcal{T}$ we have 
    \[
    \|\nabla F_{T,n}(r)\|_1\leq 1/2.
    \]
    This in turn implies that for each tree $(T,v)$ and each $r,r'\in B([\ell_\Delta(\lambda),r_\Delta(\lambda)],\eta_1)^{S^*_T(v,n)}$ we have
    \begin{equation}\label{eq:complex contraction F}
        |F_{T,n}(r)- F_{T,n}(r')| \leq 1/2 \|r-r'\|_\infty.
    \end{equation}
    
    Next we look at the perturbed maps $F_{T,n,\Vec\lambda}$.
    By continuity of $(\Vec{\lambda},r')\mapsto F_{T,n,\Vec\lambda}(r')$ and the fact that there only finitely many maps $F_{T,n}$, there exists $\eta_2>0$ such that for all trees $(T,v)\in \mathcal{T}$, $r'\in  B([\ell_\Delta(\lambda),r_\Delta(\lambda)],\eta_1)^{S^*_T(v,n)}$ and any $\Vec{\lambda}=(\lambda_u)_{u\in V(T)}\in B(\lambda,\eta_2)^{V(T)}$, we have
    \begin{equation}\label{eq:complex perturbation F}
        |F_{T,n,\Vec\lambda}(r')-F_{T,n}(r')|<\eta_1/2.
    \end{equation}
    
    Now combining~\eqref{eq:complex perturbation F} and~\eqref{eq:complex contraction F} with the triangle inequality, we obtain that for any tree $(T,v)\in \mathcal{T}$, any $r\in  [\ell_\Delta(\lambda),r_\Delta(\lambda)]^{S^*_T(v,n)}$, $r'\in B_{\ell_\infty}(r,\eta_1)$ and any  $\Vec{\lambda}=(\lambda_u)_{u\in V(T)}\in B(\lambda,\eta_2)^{V(T)}$,
    \begin{align*}
        |F_{T,n,\Vec\lambda}(r')-F_{T,n}(r)| &  \leq |F_{T,n,\Vec\lambda}(r')-F_{T,n}(r')|  +|F_{T,n}(r')-F_{T,n}(r)|
        \\
        &\leq \eta_1/2+ 1/2\|r-r'\|_\infty \leq \eta_1,
    \end{align*}
    as desired.
    %
    %
\end{proof}

We can now prove Theorem~\ref{thm:main}. 
\begin{proof}[Proof of Theorem~\ref{thm:main}]
    To prove the implication we assume that $\mathcal{G}$ is family of graphs of maximum degree at most $\Delta$, is closed under taking induced subgraphs, and satisfies VSSM at fugacity $\lambda>0$ with rate $\alpha\in [0,1)$ and constant $C>0$.  
    Let $n\in \mathbb{N}$, $\eta_1\in(0,1)$, and $\eta_2>0$ be the guaranteed constants from Proposition~\ref{prop:contraction  wrp to sup norm}.
    Consider the collection $\mathcal{T}$ of all rooted trees of maximum degree at most $\Delta$ and depth at most $n$. 
    Since $\mathcal{T}$ consists of only finitely many trees, there exists $\varepsilon\in (0,\eta_2)$ such that for each $(T,v)\in \mathcal{T}$ and any $\Vec\lambda\in B(\lambda,\varepsilon)^{V(T)}$ we have $R_{T,v}(\Vec\lambda)\in B(R_{T,v}(\lambda),\eta_1).$
    
    We will show that for all $G=(V,E)\in \mathcal{G}$ and $(\lambda_u)_{u\in V}$ such that $\lambda_u\in B(\lambda,\varepsilon)$ for all $u\in V$ we have 
    
    \begin{itemize}
        \item[(i)] $Z_G(\lambda_u)\neq 0$, and  
        \item[(ii)] $R_{G,v}(\lambda_u)\in B(R_{G,v}(\lambda),\eta_1)$ for any vertex $v\in V(G).$
    \end{itemize}
    In case $V=\emptyset$ this is trivial and thus we may assume that $|V|\geq 1$.
    We claim that it suffices to show (ii). 
    Indeed, by induction we know that $Z_{G-v}(\lambda_u)\neq 0$, and thus by~\eqref{eq:ratio -1 vs Z=0} it suffices to show that $R_{G,v}(\lambda_u)\neq -1$ for any vertex $v\in V$.
    Since $R_{G,v}(\lambda)\geq 0$ and $\eta_1<1$ it follows that (ii) implies (i).
    
    To show (ii) let $T=T_{\mathrm{SAW}(G,v)}$. 
    We will show that $R_{T,v}(\tilde\lambda_w)\in B(R_{T,v}(\lambda),\eta_1)$, which by Lemma~\ref{lem:ratio TSAW is ratio G} implies that $R_{G,v}(\lambda_u)\in B(R_{G,v}(\lambda),\eta_1)$.
    
    In case the depth of $T$ is at most $n$, we have that $R_{T,v}(\tilde\lambda_w)\in B(R_{T,v}(\lambda),\eta_1)$ by our choice of $\varepsilon$. So we may assume that the depth of $T$ is at least $n$.
    We know that
    \[
    R_{T,v}(\tilde\lambda_w)=F_{T,n,\Vec \lambda}(R_{T_u,u}(\tilde\lambda_w)),
    \]
    where for $u\in S^*_T(v,n)$, $T_u$ denotes the subtree of $T$ rooted at $u$ consisting of all vertices whose path to $v$ goes trough $u$ and where $\Vec \lambda$ denotes the vector of variables obtained from $(\tilde \lambda_w)$. Note that $T_u$ consists of at least one edge since $u\in S^*_T(v,n)$.
    By Lemma~\ref{lem:ratio TSAW is ratio G}, we know that each such subtree $T_u$ is the tree of self avoiding walks of a strict induced subgraph $H$ of $G$ and hence by induction and another application of Lemma~\ref{lem:ratio TSAW is ratio G} we obtain that  
    \[
    R_{T_u,u}(\tilde\lambda_w)=R_{H,u}(\lambda_{u'})\in B(R_{H,u}(\lambda),\eta_1)= B(R_{T_u,u}(\lambda),\eta_1)
    \]
    for each $u\in S^*_T(v,n)$. 
    Note that by Lemma~\ref{lem:decomposing F into a composition of mobius transformations} and Lemma~\ref{lem:upper bound ratio} we have $R_{T_u,u}(\lambda)\in [\ell_\Delta(\lambda),r_\Delta(\lambda)]$ and thus by Proposition~\ref{prop:contraction  wrp to sup norm} we have 
    \[
    |F_{T,n,\Vec \lambda}(R_u((\tilde \lambda_w )))-F_{T,n}(R_{T_u,u}(\lambda))| \leq \eta_1.
    \]
    Since $R_{G,v}(\lambda_u)=R_{T,v}(\tilde \lambda_w)=F_{T,n,\Vec \lambda}(R_u((\tilde \lambda_w )))$
    and $R_{G,v}(\lambda)=R_{T,v}(\lambda)=F_{T,n}(R_u(\lambda))$ by yet another application of Lemma~\ref{lem:ratio TSAW is ratio G} this shows (ii) for $(G,v)$ and finishes the proof of Theorem~\ref{thm:main}.
\end{proof}

\section{VSSM at growing distance allows zeros}\label{sec:vssm at distance}
In this section we prove Theorem~\ref{thm:VSSM at a distance and zeros}.
The family of graphs $\mathcal G$ that we consider consists of trees of the following form:

Let $d=\Delta-1 \ge 2$. For $k\in \mathbb{N}$ let $T_{d^k}$ be the rooted tree in which the root vertex has degree $d$ and is connected to $d$ identical copies of $T_{d^{k-1}}$, where $T_{d^0}$ is just a single vertex.
For $k,m\in\mathbb{N}$ let $T_{d^{k},1^{m}}$ be obtained from $T_{d^k}$ by attaching
a path (tree of down-degree 1) of depth $m$ to each leaf. 

All these trees have maximum degree $\Delta$. The zeros of the independence polynomials of trees $T_{d^{k}},k\in\mathbb{N}$ are
known to accumulate at $\lambda_{c}(\Delta)$, as shown via dynamical instability
in~\cite{PetersRegts2019OnaConjectureofSokalConcerningRootsoftheIndependencePolynomial}.
We will give a variant of this proof and show that the dynamical instability persists for the trees $T_{d^{k},1^{m}}$,
while the contracting dynamics of long paths ensures $\varphi$-VSSM.


For the rooted tree $T=T_{d^k,1^m}$, equation \eqref{eq:recursion ratio tree} simplifies to 
\[
R_{T,v}(\lambda)=\frac{\lambda}{(1+R_{T',v'})^{d}},
\]
where $v'$ is a neighbour of $v$ in $T$ and $T'$ is the component
of $T-v$ with root $v'$. In other words, $R_{T,v}$ is the image
of $R_{T',v'}$ under the rational map 
\[
f_{d,\lambda}:\mathbb{C}\to\mathbb{C},\text{ defined by } R\mapsto \frac{\lambda}{(1+R)^{d}}.
\]
By recursion, this gives us a representation of the ratio of $T_{d^{k},1^{m}}$
at the root $v_{0}$ as the composition: 
\begin{equation}
    R_{T_{d^{k},1^{m}},v_{0}}(\lambda)=f_{d,\lambda}^{\circ k}\circ f_{1,\lambda}^{\circ m}(0).\label{eq:TreeRatioAsComposition}
\end{equation}
We can thus study these ratios via the dynamical behavior of $f_{1,\lambda}$ and
$f_{d,\lambda}$.
The next subsection contains a summary of useful properties of these maps.

\subsection{Dynamics of \texorpdfstring{$f_{d,\lambda}$}{}}

\label{subsec:treeDynamics} Let us first fix some language and notation:
A rational map $f$ with real or complex coefficients defines a holomorphic
map from the extended complex plane $\widehat{\mathbb{C}}$ to itself.
A fixed point $x=f(x)$ of $f$ is attracting if $|f'(x)|<1$, repelling
if $|f'(x)|>1$, and parabolic if $f'(x)$ is a root of unity. The
basin of a fixed point $x$ of $f$ is
\[
B_{f}(x):=\{y\in\widehat{\mathbb{C}}\mid\lim_{n\to\infty}f^{n}(y)=x\}.
\]
If $x$ is repelling, there can be no orbit converging to $x$ non-trivially,
i.e. that is not equal to $x$ after finitely many iterates.

The dynamics of $f_{d,\lambda}$ have been studied in \cite{PetersRegts2019OnaConjectureofSokalConcerningRootsoftheIndependencePolynomial},
where the first two authors show in particular:
\begin{lem}
    Let $d\in\mathbb{N}$. Then:
    \begin{enumerate}
        \item For $\lambda>0$, $f_{d,\lambda}$ has a fixed point $x_{d}(\lambda)\in(0,\lambda)$ and $x_{d}(\lambda)$ is attracting for $\lambda<\lambda_{c}(d+1)$,
        parabolic for $\lambda=\lambda_{c}(d+1)$, and repelling for $\lambda>\lambda_{c}(d+1)$.
        \item The fixed point $x_{d}(\lambda)$ persists for $\lambda$ in a complex
        neighbourhood of $[0,\infty)$ and depends holomorphically on $\lambda$.
    \end{enumerate}
\end{lem}

\begin{note}
    For $d=1$, $\lambda_{c}(d+1)=+\infty$, so $x_{1}(\lambda)$
    is attracting for all $\lambda>0$.
\end{note}

We can further understand the limit behaviour of orbits on the entire
interval $[-1,\infty]$:
\begin{lem}\label{lem:FixedPointsfsquared}
    Let $d\in\mathbb{N}_{>0}$ and $\lambda>0$. Then $x_d(\lambda)$ is the unique fixed point of $f_{d,\lambda}$ in the interval $[-1,\infty]$ and:
    \begin{enumerate}
        \item If $\lambda\le\lambda_{c}(d+1)$, then the entire interval $[-1,\infty]$
        lies in the basin of the attracting or parabolic fixed point $x_{d}(\lambda)$.
        \item If $\lambda>\lambda_{c}(d+1)$, then $f_{d,\lambda}$ has an attracting
        $2$-cycle $(x_{1},x_{2})=(x_{d,1}(\lambda),x_{d,2}(\lambda))\in(0,\lambda)^{2}$ (i.e. $x_{1}$ and $x_{2}$
        are attracting fixed points of $f_{d,\lambda}^{2}$ with $f_{d,\lambda}(x_{1})=x_{2}$)
        and the entire interval $[-1,\infty]$ except the fixed point $x_{d}(\lambda)$
        lies inside the basin of the attracting cycle.
    \end{enumerate}
\end{lem}

\begin{rem}\label{rem:stabilityWRTparameter}
    The convergence in both cases is uniform on compact subsets of the basin. If $x_{d}(\lambda)$ is attracting, convergence
    to $x_{d}(\lambda)$ is locally uniform in both $x$ and $\lambda$.
    More precisely, for each $\lambda<\lambda_{c}(d+1)$ and every $x$
    in the basin $B_{f_{d,\lambda}}(x_{d}(\lambda))$ there exist neighbourhoods
    $\Lambda\subseteq\mathbb{C}$ of $\lambda$ and $U\subseteq\widehat{\mathbb{C}}$
    of $x$ such that
    \[
    \lim_{n\to\infty}f_{d,\lambda'}^{n}(x')=x_{d}(\lambda')
    \]
    uniformly in $(\lambda',x')\in\Lambda\times U$ (see \cite{Milnor2006Dynamicsinonecomplexvariable}
    and \cite{ManeSadSullivan1983OntheDynamicsofRationalMaps}).
\end{rem}

\begin{proof}
    For any $\lambda>0$, $f_{d,\lambda}$ is strictly decreasing on $[-1,+\infty]$,
    so $g\coloneqq f_{d,\lambda}^{2}$ is strictly increasing, and hence
    injective. We know $g([-1,\infty])=[0,\lambda]$, so
    \begin{equation}
        \bigcap_{n\in\mathbb{N}}g^{n}([-1,\infty])=[x_{1},x_{2}],\label{eq:NestedIntervalsforfsquared}
    \end{equation}
    for some $0<x_{1}\le x_{2}<\lambda$.
    By monotonicity, $x_{1},x_{2}$ are fixed points of $g$
    that are either attracting or parabolic and $f_{d,\lambda}(x_{1})=x_{2}$.
    
    We recall that every non-repelling
    fixed point of a rational map is related to the orbit of a critical
    point (see \cite{Milnor2006Dynamicsinonecomplexvariable}). The critical
    points of $g$ are $\infty$ and $-1$ and are contained in the basins
    of $x_{1}$ and $x_{2}$ respectively, so all further fixed points
    of $g$ must be repelling. Moreover, any fixed points of $g$ in $[-1,\infty]$
    must be contained in $[x_{1},x_{2}]$, including the fixed point $x_{0}=x_{d}(\lambda)$
    of $f_{d,\lambda}$.
    
    (1) Let now $\lambda\le\lambda_{c}(d+1)$. Then the fixed point $x_{0}$
    is attracting or parabolic for $f_{d,\lambda}$ and hence for $g$,
    so $x_{0}\in\{x_{1},x_{2}\}$, but since $f_{d,\lambda}$ maps $x_{1}$
    and $x_{2}$ to each other, $g$ has the unique fixed point $x_{0}=x_{1}=x_{2}$
    on $[-1,\infty]$ and, by \eqref{eq:NestedIntervalsforfsquared},
    the entire interval $[-1,\infty]$ is in the basin of $x_{0}$.
    
    (2) Let $\lambda>\lambda_{c}(d+1)$. Then $x_{0}$ is repelling, so
    $x_{1}<x_{0}<x_{2}$ and if $x>x_{0}$ is small enough, we
    have $g(x)>x$. As before, the monotonicity of $g$ then implies that
    \[
    \bigcap_{n\in\mathbb{N}}g^{n}([x,x_{2}])=[x_{3},x_{2}]
    \]
    with $x_{3}=\lim_{n\to\infty}g^{n}(x)\in(x_{0},x_{2}]$, which is
    an attracting or parabolic fixed point of $g$, so $x_{3}$ must be
    $x_{2}$. Thus we have shown that $\lim_{n\to\infty}g^{n}(x)=x_{2}$
    locally uniformly for all $x\in(x_{0},\infty]$. An analogous argument
    shows that $\lim_{n\to\infty}g^{n}(x)=x_{1}$ locally uniformly for
    all $x\in[-1,x_{0})$.
    
    We have thus shown that all orbits in $[-1.,\infty]$ except the repelling fixed point $x_0$ converge to either the fixed point $x_1$ or $x_2$ and these are either both attracting or both parabolic. We recall that, for every orbit converging to a parabolic fixed point, there must also be an orbit of a critical point converging to it from the same direction (see \cite{Milnor2006Dynamicsinonecomplexvariable}). $x_1$ and $x_2$ each have orbits converging to them from above and from below, but $g$ only has two critical points, so $x_1$ and $x_2$ must be attracting fixed points.
\end{proof}
To demonstrate the relation of zero-freeness and dynamical stability,
we use Lemma~\ref{lem:FixedPointsfsquared} to give a short proof
of a weaker version of a result in \cite{PetersRegts2019OnaConjectureofSokalConcerningRootsoftheIndependencePolynomial}:
\begin{lem}\label{lem:zeroFreeRegion}
    The independence polynomials of the family $\mathcal{T}=\{T_{d^{k}}\}_{k}$
    have no zeros on a neighbourhood of $[0,\lambda_{c}(d+1))$, but have
    zeros accumulating at $\lambda_{c}(d+1)$.
\end{lem}

The proof uses the notion of normal families and Montel's Theorem, which are fundamental notions in the study of complex dynamical systems, see for example \cite{Beardon1991IterationofRationalFunctions},\cite{CarlesonGamelin1993ComplexDynamics}, \cite{Milnor2006Dynamicsinonecomplexvariable}. Normality is a notion of dynamical stability defined as follows:

\begin{defn}
    A family of holomorphic functions $\mathcal{F}$ is normal at $\lambda\in\mathbb{C}$,
    if there exists a neighbourhood $U$ of $\lambda$ such that for each
    sequence $(g_{n})_{n}$ in $\mathcal{F}$ there exists a subsequence
    $(g_{n_{k}})_{k}$ that converges uniformly on $U$ to a holomorphic
    map $g:U\to\widehat{\mathbb{C}}$.
\end{defn}

Montel's theorem relates normality of a family to omitted values:
\begin{thm}[Montel]
    \label{thm:Montel} Let $\mathcal{F}$ be a family of holomorphic
    functions from an open set $U\subseteq\mathbb{C}$ to the Riemann
    sphere $\mathbb{C}\cup\{\infty\}$ that all omit three values in $\mathbb{C}\cup\{\infty\}$.
    Then $\mathcal{F}$ is normal on $U$.
\end{thm}

We can now prove the lemma.

\begin{proof}[Proof of Lemma~\ref{lem:zeroFreeRegion}]
    By Remark~\ref{rem:stabilityWRTparameter}, convergence of the sequence
    of ratios $(R_{T_{d^{n}}}(\lambda)=f_{d,\lambda}^{n}(0))_{n}$ to
    $x_{d}(\lambda)$ persists on a small complex neighbourhood of $[0,\lambda_{c}(d+1))$.
    In particular none of these ratios are $-1$, so there are no zeros
    of the independence polynomial in that neighbourhood.
    
    
    Assume that $R_{T_{d^{n}}}(\lambda)=f_{d,\lambda}^{n}(0)$ is not
    equal to $-1$ for all $n\in\mathbb{N}$ and all $\lambda$ in a neighbourhood
    $\Lambda$ of $\lambda_{c}(d+1)$. Then, since $x=-1$ is the only
    solution to $f_{d,\lambda}(x)=\infty$ and $\infty$ is the only solution
    to $f_{d,\lambda}(x)=0$, the family $\{R_{T_{d^{n}}}(\lambda)=f_{d,\lambda}^{n}(0)\}_{n}$
    avoids $-1,$ $\infty$, and $0$. Hence, by Montel's theorem~\ref{thm:Montel},
    the family is normal on $\Lambda$.
    
    However, the change of dynamics in Lemma~\ref{lem:FixedPointsfsquared} shows
    that the family $\mathcal{F}=\{R_{T_{d^{n}}}(\lambda)=f_{d,\lambda}^{n}(0)\}_{n}$
    cannot be normal near $\lambda_{c}(d+1)$: Assume there is a neighbourhood $\Lambda$ of $\lambda_{c}(d+1)$ on which
    $\mathcal{F}$ is normal. We may assume $\Lambda$ to be connected. Then, by definition, for the sequences $(f_{d,\lambda}^{2n})_{n}$
    and $(f_{d,\lambda}^{2n+1})_{n}$ there exist subsequences
    $(n_{k})_{k}$ and $(n_{k}')_{k}$ with holomorphic limit functions:
    \[
    g_{1}(\lambda)=\lim_{k\to\infty}f_{d,\lambda}^{2n_{k}}(0),\quad g_{2}(\lambda)=\lim_{k\to\infty}f_{d,\lambda}^{2n_{k}'+1}(0).
    \]
    By Lemma~\ref{lem:FixedPointsfsquared}, we have
    \[
    g_{1}(\lambda)=x_{d}(\lambda)=g_{2}(\lambda)\quad\text{for }\lambda\le\lambda_{d},
    \]
    but
    \[
    g_{1}(\lambda)=x_{d,1}(\lambda)\neq x_{d,2}(\lambda)=g_{2}(\lambda)\quad\text{for }\lambda>\lambda_{d}.
    \]
    This is a contradiction to the identity theorem for the holomorphic
    functions $g_{1}$ and $g_{2}$ on $\Lambda$.
    
    Hence, for any neighbourhood $\Lambda$ of $\lambda_{c}(d+1)$
    there exist $\lambda\in\Lambda$ and $T_{d^{n}}\in\mathcal{T}$, such
    that $R_{T_{d^{n}}}(\lambda)=f_{d,\lambda}^{n}(0)=-1$, that is, $Z_{T_{d^{n}}}(\lambda)=0$.
\end{proof}

We will show that this dynamical derivation of accumulating zeros
can still be applied when we attach long paths to the leaves.

\subsection{Zeros accumulating at \texorpdfstring{$\lambda_{c}(\Delta)$}{the critical point}}
For a sequence of positive integers $(m_{k})_{k}$, we denote $T_{k}=T_{d^{k},1^{m_{k}}}$.
Then we have
\begin{equation}
    R_{T_{k},v_{0}}=f_{d,\lambda}(R_{T_{k}',v'})=f_{d,\lambda}^{2}(R_{T_{k}'',v''})
    \label{eq:twoRecursionSteps}
\end{equation}
with $T_{k}'=T_{d^{k-1}1^{m_{k}}}$ and $T_{k}''=T_{d^{k-2}1^{m_{k}}}$. The following proposition gives us one part of Theorem~\ref{thm:VSSM at a distance and zeros}.
\begin{prop}
    \label{prop:ZerosForAdjFamily} 
    Let $(m_{k})_{k}$ be a sequence of non-negative integers.
    Then at least one of the families $\mathcal{T}=\{T_{k}\}_{k}$, $\mathcal{T}_{1}=\{T_{k}'\}_{k}$,
    and $\mathcal{T}_{2}=\{T_{k}''\}_{k}$ has zeros of its independence
    polynomial accumulating at $\lambda_{c}(\Delta)$.
\end{prop}

\begin{proof}
    Our strategy is the same as for the homogeneous case above: Assume
    for contradiction, $Z_{T_{k}}(\lambda)$, $Z_{T_{k}'}(\lambda)$,
    and $Z_{T_{k}''}(\lambda)$ are non-zero for every $\lambda$ in a
    neighbourhood $\Lambda$ of $\lambda_{d}$ and all $k\in\mathbb{N}$.
    Then $R_{T_{k},v}$, $R_{T_{k}',v'}$, and $R_{T_{k}'',v''}$ all
    omit the value $-1$ on $\Lambda$, and, by \eqref{eq:twoRecursionSteps},
    $R_{T_{k},v}$ omits $-1$, $\infty=f_{d,\lambda}(-1)$ and $0=f_{d,\lambda}^{2}(-1)$.
    This would imply that $R_{T_{k},v_{0}}(\lambda)$ is normal on $\Lambda$
    by Theorem~\ref{thm:Montel} (Montel's theorem).
    
    To reach a contradiction, we therefore need to show that the family
    $\{R_{T_{k},v_{0}}\}_{k}$ cannot be normal near $\lambda_{c}(\Delta)$. Recall
    that
    \[
    R_{T_{k},v_{0}}(\lambda)=f_{d,\lambda}^{\circ k}\circ f_{1,\lambda}^{\circ m_{k}}(0)
    \]
    and, since $0$ is contained in the attracting basin of the fixed
    point $x_{1}(\lambda)$ of $f_{1,\lambda}$ for all $\lambda$, we
    know that $f_{1,\lambda}^{m_{k}}(0)$ converges
    to $x_{1}(\lambda)$ uniformly on a neighbourhood of $\lambda_{c}(\Delta)$.
    Moreover, for $\lambda\in(0,\lambda_{c}(\Delta)]$, the limit $x_{1}(\lambda)$
    is real and hence contained in the basin of $x_{d}(\lambda)$. Taking
    $z(\lambda)=x_{1}(\lambda)$ and $z_k({\lambda})=f_{1,\lambda}^{m_{k}}(0)$
    in Lemma~\ref{lem:attractingPlusNonNormalIsNormal} below, shows
    that $\{R_{T_{k},v_{0}}\}_{k}$ cannot be normal, completing our proof
    by contradiction.
\end{proof}

\begin{rem}
    Removing the root (and its neighbours) to create the families $\mathcal T'$ (and $\mathcal T''$) enabled our argument similar to~\cite[Lem.~13]{Buys2021CayleyTreesDoNotDeterminetheMaximalZeroFreeLocusoftheIndependencePolynomial}. In fact, a more technical proof using the Theorem of Lavaurs \cite{Lavaurs1989SystemesDynamiquesHolomorphesExplosionDePointsPeriodiquesParaboliques} (see also \cite{Douady1994DoesaJuliaSetDependContinuouslyonthePolynomial}) shows that it is enough to only consider the family $\mathcal T$ (see the appendix of \cite{PetersRegts2019OnaConjectureofSokalConcerningRootsoftheIndependencePolynomial}).
    Since the above result is sufficient for our purposes and its proof is much more elementary, we leave the proof of the more elegant statement to the interested reader.
\end{rem}

\begin{lem}\label{lem:attractingPlusNonNormalIsNormal}
    Let $\Delta\in\mathbb N$ and let $z(\lambda)$ be a holomorphic function in $\lambda$ on a neighbourhood of $\lambda_c(\Delta)$. Assume that $z(\lambda) \neq x_{d}(\lambda)$ for all $\lambda$ and that, moreover, for non-negative real $\lambda\le\lambda_c(\Delta)$, $z(\lambda)$ is real and contained in the basin of the attracting or parabolic fixed point $x_{d}(\lambda)$. Moreover, let $(z_k({\lambda}))_k$ be a sequence of holomorphic functions in $\lambda$ such that $z_k({\lambda})\xrightarrow[k\to\infty]{}z(\lambda)$
    uniformly in $\lambda$ near $\lambda_{c}(\Delta)$. Then the family $\{f_{d,\lambda}^{n}(z_n(\lambda))\}_{n}$
    is not normal in $\lambda$ near $\lambda_{c}(\Delta)$.
\end{lem}

\begin{proof}
    For fixed $\lambda\le\lambda_{c}(\Delta)$ and $n$ large enough, $z_n(\lambda)$
    is contained in a compact subset $K_{\lambda}$ of the basin of $x_{d}(\lambda)$,
    and by uniform convergence we have
    \[
    \lim_{n\to\infty}f_{d,\lambda}^{n}(z_n(\lambda))=x_{d}(\lambda).
    \]
    For $\lambda>\lambda_{c}(\Delta)$, $z(\lambda)$ remains real by the identity theorem for holomorphic functions and if $\lambda$ is small enough, since $z({\lambda_{d}}) \neq x_{d}(\lambda_{d})$,
    we still have $z(\lambda)\in [-1,\infty]\setminus \{x_{d}(\lambda)\}$, so $z(\lambda)$
    is in the basin of the attracting periodic cycle $(x_{d,1}(\lambda),x_{d,2}(\lambda))$
    of $f_{d,\lambda}$ and we have
    \[
    f_{d,\lambda}^{2n}\to x_{d,1}(\lambda),\quad f_{d,\lambda}^{2n+1}\to x_{d,2}(\lambda)
    \]
    uniformly on a compact neighbourhood $K_{\lambda}$ of $z(\lambda)$.
    
    Assume $\{f_{d,\lambda}^{n}(z_n(\lambda))\}_{n}$ is normal on a
    connected neighbourhood $\Lambda$ of $\lambda_{c}(\Delta)$. Then there exists
    holomorphic limit functions
    \[
    g_{1}(\lambda)=\lim_{k\to\infty}f_{d,\lambda}^{2n_{k}}(z_n(\lambda)),\quad g_{2}(\lambda)=\lim_{k\to\infty}f_{d,\lambda}^{2n_{k}+1}(z_n(\lambda))
    \]
    with
    \[
    g_{1}(\lambda)=x_{d}(\lambda)=g_{2}(\lambda)\quad\text{for }\lambda<\lambda_{d},
    \]
    but
    \[
    g_{1}(\lambda)=x_{d,1}(\lambda)\neq x_{d,2}(\lambda)=g_{2}(\lambda)\quad\text{for }\lambda>\lambda_{d}.
    \]
    This is a contradiction to the identity theorem for the holomorphic
    functions $g_{1}$ and $g_{2}$ on $\Lambda$.
\end{proof}

\subsection{\texorpdfstring{$\varphi$-VSSM for $\{T_{k}\}_k$}{VSSM for the trees Tk}}
Given an unbounded increasing function $\varphi$, we now show that if we chose the sequence $(m_{k})_{k}$ growing fast enough,
the family of trees $\mathcal{T}=\{T_{k}=T_{d^{k},1^{m_{k}}}\}_{k\in\mathbb{N}}$
satisfies $\varphi$-VSSM for all parameters $\lambda\in[0,\lambda^*]$, allowing us to prove Theorem~\ref{thm:VSSM at a distance and zeros}.

\begin{prop}
    \label{prop:SSMestimateFixedk}
    Let $\lambda^\star>0$. For each $k\in\mathbb{N}$, there exists
    a constant $M_{k}>0$ such that for all $\lambda\in[0,\lambda^\star]$, all induced subgraphs $H$ of $T_{k}$, and
    all $\ell>2k$ we have
    \begin{equation}
        |R_{(H,v),\tau_{1}}(\lambda)-R_{(H,v),\tau_{2}}(\lambda)|<M_{k}c^{\ell}\label{eq:SSMestimateFixedk}
    \end{equation}
    for all $v\in V(H)$ and all boundary conditions $\tau_{1},\tau_{2}:S_H(v,\ell)\to\{0,1\}$.
\end{prop}

Before the proof, let us show how Proposition~\ref{prop:SSMestimateFixedk} together with Proposition~\ref{prop:ZerosForAdjFamily} implies Theorem~\ref{thm:VSSM at a distance and zeros}.

\begin{proof}[Proof of Theorem~\ref{thm:VSSM at a distance and zeros}]
    Fix a function $\varphi$. 
    Let $\mathcal{G}=\{T_{d^k,1^{m_k}}\}_k\cup \{T_{d^{k-1},1^{m_k}}\}_k\cup \{T_{d^{k-2},1^{m_k}}\}_k$ for some increasing unbounded sequence $(m_k)$ to be determined below.
    
    By Proposition~\ref{prop:ZerosForAdjFamily} we know that zeros of the independence polynomial of graphs in $\mathcal{G}$ accumulate at $\lambda_{c}(\Delta)$ for any such sequence $(m_k)$.
    So it suffices to show that there exists a sequence $(m_k)$ for which $\mathcal{G}$ satisfies $\varphi$-VSSM at any $\lambda\in[0,\lambda^*]$.
    
    Let $c<c_{1}<1$ and $M_{k},k\in\mathbb{N}$ from Proposition~\ref{prop:SSMestimateFixedk}.
    Then for each $k\in\mathbb{N}$, there exists an $\ell_{k}>2k$ such
    that for all $\ell\ge\ell_{k}$, we have $M_{k}c^{\ell}<c_{1}^{\ell}$
    for all $\ell\ge\ell_{k}$. Now take $m_{k}$ large enough that
    \[
    \varphi(|V(T_{d^{k-2}1^{m_{k}}})|)\ge\ell_{k}.
    \]
    Then for every $T=(V,E)\in\{T_{d^k,1^{m_k}}\}_k$, every induced subgraph $H$ of $T$,
    we have
    \[
    |R_{(H,v),\tau_{1}}(\lambda)-R_{(H,v),\tau_{2}}(\lambda)|<c_{1}^{\ell}
    \]
    for all $v\in V(H$), all boundary conditions $\tau_{1},\tau_{2}:S_H(v,\ell)\to\{0,1\}$ with $\ell\geq \varphi(|V(T_{d^{k}1^{m_{k}}})|)$ and $\lambda\in[0,\lambda^*]$.
    In other words, the hard-core model on $\mathcal{T}=\{T_{d^{k}1^{m_{k}}}\}_{k}$  satisfies $\varphi$-VSSM at any $\lambda\in[0,\lambda^*]$.
    Note that the same is true
    for the families $\mathcal{T}'=\{T_{d^{k-1},1^{m_{k}}}\}_{k}$ and
    $\mathcal{T}''=\{T_{d^{k-2},1^{m_{k}}}\}_{k}$.
    This finishes the proof.
\end{proof}
We now prove the proposition. We note that an alternative proof can be derived from bounds on the connective constant following the arguments in \cite{SinclairSrivastavaStefankovicYin2017SpatialMixingandtheConnectiveConstantOptimalBounds}.
\begin{proof}[Proof of Proposition~\ref{prop:SSMestimateFixedk}]
    Fix $k\in\mathbb{N}$. Recall that for a boundary condition $\tau:S_T(v,\ell)\to\{0,1\}$
    on a tree $T$, we still have a recursion formula:
    \[
    R_{(T,v),\tau}=\frac{\lambda}{\prod_{j=1}^{n}(1+R_{(T_{j},v_{j}),\tau})},
    \]
    where $v_{1},\ldots,v_{n}$ are the neighbours of $v$ in $T$ and
    $T_{j}$ is the component of $T-v$ containing $v_{j}$ for $j=1,\ldots,n$.
    
    Let $v$ be a vertex in a connected induced subgraph $H$ of $T_{k}$
    and $T_{H}=H\cap T_{d^{k}}$. We ensure the estimate \eqref{eq:SSMestimateFixedk}
    by the following heuristic in each case:
    \begin{enumerate}
        \item [(1)]$v\in T_{H}$. Then for $\ell\gg2k$, $B_{\ell}(v)$ contains
        $T_{H}$ and a long part of each attached path. The difference of
        the boundary conditions is reduced by the contraction under each $f_{1,\lambda}$
        along each long path. (See Figure~\eqref{fig:TreeBoundaryLocations},
        left)
        \item [(2a)]$v\notin T_{H}$, but close to $T_{H}$. For $\ell$ large
        enough, $B_{\ell}(v)$ contains $T_{H}$ and a long part of each attached
        path. As before, the difference is reduced along the paths not containing
        $v$ until reaching $T_{H}$, where a bounded expansion may happen
        and then not expanded by the path from $T_{H}$ to $v$. (See Figure~\eqref{fig:TreeBoundaryLocations},
        right, $m_{j}$ large)
        \item [(2b)]$v\notin T_{H}$ and far away from $T_{H}$. If $T_{H}$ is
        far enough from $v$, then the path from $T_{H}$ to $v$ leads to
        enough iterations of $f_{1,\lambda}$ that any difference is reduced
        sufficiently. (See Figure~\eqref{fig:TreeBoundaryLocations}, right,
        $m_{1}$ large)
    \end{enumerate}
    \begin{figure}
        \includegraphics{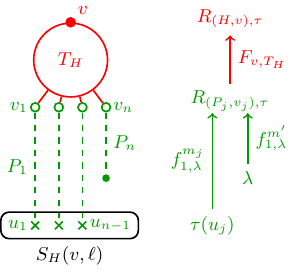}\qquad{}\includegraphics{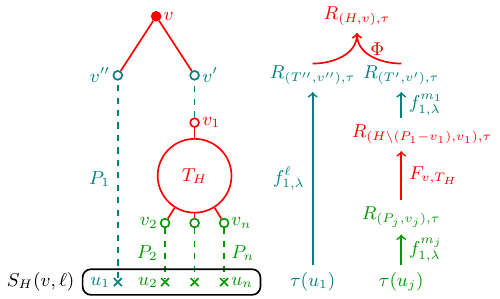}
        \caption{How ratios propagate upwards in the tree towards a root vertex $v$
            in $T_{H}$ (left) and in $P_{1}$ (right). Shown is $B_{\ell}(v)\subseteq H$. $\tau$ is a boundary condition on $S_H(v,\ell)$. In the left picture, $P_n$ examplifies a path that does not intersect the boundary $S_H(v,\ell)$, so $\tau$ has no influence along this path.}
        \label{fig:TreeBoundaryLocations}
    \end{figure}
    
    Let us quantify the maximal gain of correlation passing through $T_{H}$ for given $\lambda \in [0,\lambda^*]$.
    Let $\mathcal{L}_{T_H}=\{v_{1},\ldots v_{n}\}$ be the set of leaves
    of $T_{H}$. Let $P_{1},\ldots,P_{n}$ be the paths attached to the
    respective leaves $v_{1},\ldots v_{n}$ of $T_{H}$ to complete $H$.
    \begin{enumerate}
        \item Let $v$ in $T_{H}$. Then there is a rational map
        \[
        F_{v,T_{H},\lambda}:[0,\infty]^{\mathcal{L}_{T_H}\backslash\{v\}} \to[0,\infty],
        \]
        depending analytically on $\lambda$, that maps $(R_{(P_{j},v_{j}),\tau})_{v_{j}\neq v},$ to $R_{(H,v),\tau}$
        for any boundary condition $\tau$ on $S_H(v,\ell)$ with $\ell>2k$.
        (Note that $v$ may be equal to a leaf $v_{j}$ for some $j$.)
        \item Let $v$ in $P_{\alpha}$, $v\neq v_{\alpha}$. Then there exists
        a rational map
        \[
        F_{v,T_{H},\lambda}=F_{P_{\alpha},T_{H},\lambda}:[0,\infty]^{\mathcal{L}_{T_H}\backslash\{v_{\alpha}\}}\to[0,\infty],
        \]
        depending analytically on $\lambda$, that maps $(R_{(P_{j},v_{j}),\tau})_{v_{j}\neq v_{\alpha}}$ to
        $R_{(H\backslash(P_{\alpha}-v_{\alpha}),v_{\alpha}),\tau}$ for any
        boundary condition $\tau$ on $S_H(v,\ell)$ such that $S_H(v,\ell)$
        does not contain vertices of $T_{H}$.
    \end{enumerate}
    
    Let $c\in(0,1)$ be a constant such that $|f_{1,\lambda}'(x_{0}(\lambda))|<c$
    for all $\lambda\in[0,\lambda^{*}]$, where we recall that $x_{1}(\lambda)$
    denotes the attracting fixed point of $f_{1,\lambda}$ in $\mathbb{R}^{+}$.
    Since $[0,\lambda^{*}]$ is compact, there exist uniform neighbourhoods
    $U_{\lambda}=B_{r}(x_{1}(\lambda))$ of $x_{1}(\lambda)$ for some
    fixed $r>0$ for all $\lambda\in[0,\lambda^{*}]$, such that for all
    $x,y\in U_{\lambda}$, we have $f_{1,\lambda}(x),f_{1,\lambda}(y)\in U_{\lambda}$
    and
    \begin{equation}
        |f_{1,\lambda}(x)-f_{1,\lambda}(y)|<c|x-y|.\label{eq:contractionOnBasin}
    \end{equation}
    The compact interval $I=[0,\lambda^{*}]$ is contained in the basin
    $B_{\lambda}$ of the attracting fixed point $x_{1}(\lambda)$ of
    $f_{1,\lambda}$ for every $\lambda\in[0,\lambda^{*}]$. By Remark~\ref{rem:stabilityWRTparameter},
    there is a uniform number $m_{0}\in\mathbb{N}$ of iterations of $f_{1,\lambda}$
    that maps the interval $I$ into $U_{\lambda}$, that is $f_{1,\lambda}^{m_{0}}(I)\subseteq U$
    for all $\lambda\in[0,\lambda^{*}]$. Iterating the estimate \eqref{eq:contractionOnBasin},
    it follows:
    \[
    |f_{1,\lambda}^{m}(x)-f_{1,\lambda}^{m}(y)|<|x-y|c^{m}\begin{cases}
        \frac{\norm{f_{1,\lambda}'}_{[0,\lambda]}^{m}}{c^{m}}, & \text{ if }m\le m_{0}\\
        \frac{\norm{f_{1,\lambda}'}_{[0,\lambda]}^{m_{0}}}{c^{m_{0}}}, & \text{ if }m>m_{0}
    \end{cases}
    \]
    for all $x,y\in I$, $\lambda\in[0,\lambda^{*}]$, and $m\ge0$.
    
    Therefore, there exists a constant $M_{1}>1$ such that
    \begin{equation}
        |f_{1,\lambda}^{m}(x)-f_{1,\lambda}^{m}(y)|<M_{1}c^{m}|x-y|\label{eq:ContractionRelative}
    \end{equation}
    for all $x,y\in I$, $\lambda\in[0,\lambda^{*}]$, and $m\ge0$.
    Since $f_{1,\lambda}([0,\infty])=[0,\lambda]\subseteq I$,
    there moreover exists a constant $M_{2}>1$ such that
    \begin{equation}
        |f_{1,\lambda}^{m}(x)-f_{1,\lambda}^{m}(y)|<M_{2}c^{m}\label{eq:ContractionAbsolute}
    \end{equation}
    for all $x,y\in[0,\infty]$, $\lambda\in[0,\lambda^{*}]$ and $m\ge1$.
    
    Note that if $v_{j}\notin S_H(v,\ell)$, then $R_{(P_{j},v_{j}),\tau}$
    is in the image $f_{1,\lambda}([0,\infty])=[0,\lambda]\subseteq I$. So the input
    values of each $F_{v,T,\lambda}$ are actually from the compact $I^{\mathcal{L}_{T_{H}}}$.
    Since $F_{v,T,\lambda}(x)$ is continuously differentiable in $(x,\lambda)$
    on a neighbourhood of $I^{\mathcal{L}_{T_{H}}}\times[0,\lambda^{*}]$
    and the family $\{F_{v,T,\lambda}\}_{v,T}$ is finite (finitely many
    subtrees $T$ of $T_{d^{k}}$ with finitely many vertices $v$ of
    $T$), there is a uniform bound $L>1$ on the $\sup$-norm of the
    derivative $F_{v,T,\lambda}'$ for all $v,T$ and $\lambda\in[0,\lambda^{*}]$.
    
    Now we can combine this with the decay of correlation along long paths:
    Let $v$ be a vertex of $H$, $\ell>2k$, and $\tau_{1},\tau_{2}$
    be two boundary conditions on $S_H(v,\ell)$, and $\lambda\in[0,\lambda^*]$. We again consider the three separate cases:
    \begin{enumerate}
        \item $v\in T_{H}$. Then $T_{H}\subseteq B_{\ell}(v)$, since $\ell\ge2k$.
        For each $j$, if $S_H(v,\ell)\cap P_{j}=\emptyset$, then $R_{(P_{j},v_{j}),\tau_{1}}=R_{(P_{j},v_{j}),\tau_{2}}$.
        Otherwise, let $\{u_{j}\}=S_H(v,\ell)\cap P_{j}$. Then by~\eqref{eq:ContractionAbsolute},
        we have
        \[
        |R_{(P_{j},v_{j}),\tau_{1}}-R_{(P_{j},v_{j}),\tau_{2}}|=|f_{1,\lambda}^{m_{j}}(\tau_{1}(u_{j}))-f_{1,\lambda}^{m_{j}}(\tau_{2}(u_{j}))|<M_{2}c^{m_{j}},
        \]
        where $m_{j}=d(T_{H},u_{j})>\ell-2k$. By the derivative bound, we
        have
        \[
        |R_{(H,v),\tau_{1}}-R_{(H,v),\tau_{2}}|<LM_{2}c^{\ell-2k}.
        \]
        \item[(2a)] $v\in P_{1}$ and $T_{H}\subseteq B_{\ell}(v)$. Let $m_{1}=d(v,T_{H})$
        and $m_{2}$ the minimal number of vertices of one of the paths $P_{j},j\ge2$
        in $B_{\ell}(v)$. Again, for each $j$, if $S_H(v,\ell)\cap P_{j}=\emptyset$,
        then $R_{(P_{j},v_{j}),\tau_{1}}=R_{(P_{j},v_{j}),\tau_{2}}$. Otherwise,
        let $\{u_{j}\}=S_H(v,\ell)\cap P_{j}$. Then by \eqref{eq:ContractionAbsolute},
        we have
        \[
        |R_{(P_{j},v_{j}),\tau_{1}}-R_{(P_{j},v_{j}),\tau_{2}}|=|f_{1,\lambda}^{m_{j}}(\tau_{1}(u_{j}))-f_{1,\lambda}^{m_{j}}(\tau_{2}(u_{j}))|<M_{2}c^{m_{2}},
        \]
        where $m_{j}=d(T_{H},u_{j})>m_{2}$. By the derivative bound, we then
        have
        \[
        |R_{(H\backslash(P_{1}-v_{1}),v_{1}),\tau_{1}}-R_{(H\backslash(P_{1}-v_{1}),v_{1}),\tau_{2}}|<M_{2}Lc^{m_{2}}.
        \]
        Let $v',v''$ be the neighbours of $v$, where $v'$ is the neighbour
        of $v$ on the path from $v$ to $T_{H}$ and let $T',T''$ be the
        corresponding components of $H-v$. Then, by \eqref{eq:ContractionRelative},
        we have
        \[
        |R_{(T',v'),\tau_{1}}-R_{(T',v'),\tau_{2}}|<M_{1}c^{m_{1}}\cdot M_{2}Lc^{m_{2}}<2M_{1}M_{2}Lc^{\ell-2k}
        \]
        and by the same arguments as in the previous case:
        \[
        |R_{(T'',v''),\tau_{1}}-R_{(T'',v''),\tau_{2}}|<2M_{2}c^{\ell}<M_{1}M_{2}Lc^{\ell-2k}.
        \]
        The final ratio $R_{(H,v),\tau_{j}}$ is obtained from the two ratios
        above through the rational map $\Phi_{\lambda}(x,y)=\frac{\lambda}{(1+x)(1+y)}$,
        so
        \[
        |R_{(H,v),\tau_{1}}-R_{(H,v),\tau_{2}}|<L'M_{1}M_{2}Lc^{\ell-2k}
        \]
        where $L'>1$ is a bound on the derivative of $\Phi_{\lambda}$ on
        $I^{2}$ that is uniform in $\lambda\in[0,\lambda^{*}]$.
        \item[(2b)] $v\in P_{1}$ and $T_{H}\not\subseteq B_{\ell}(v)$. Then there is
        a path of length at least $\ell-2k$ from $v$ to $T_{H}$ and by
        \eqref{eq:ContractionAbsolute}, we have
        \[
        |R_{(T',v'),\tau_{1}}-R_{(T',v'),\tau_{2}}|<M_{2}c^{\ell-2k}.
        \]
        Now we can continue as in the previous case.
    \end{enumerate}
    In conclusion, there exists a constant $M_{k}>0$ such that for $\ell>2k$,
    we have
    \[
    |R_{(H,v),\tau_{1}}-R_{(H,v),\tau_{2}}|<M_{k}c^{\ell}
    \]
    for all $H$, $v$, $\tau_{1},\tau_{2}$, and $\lambda$ as required.
\end{proof}

\begin{rem} The results in this section can be generalized by replacing the paths (trees of down-degree $1$) by trees of arbitrary down-degree $d_2$ with $1< d_2 < \Delta-1$. In this case the same arguments imply that zeros accumulate at $\lambda_c(\Delta)$, while $\varphi$-VSSM holds for $\lambda < \lambda_c(\Delta_2)>\lambda_c(\Delta)$, where $\Delta_2 = d_2+1$. The failure of $\varphi$-VSSM for $\lambda>\lambda_c(\Delta_2)$ follows
    from the emergence of the attracting $2$-cycle of $f_{d_{2},\lambda}$ (see Section~\ref{subsec:treeDynamics}), which does not occur for $d_2 = 1$.
\end{rem}

\section{Concluding remarks}\label{sec:conclusion}
We have introduced the notion of VSSM in this paper and shown that VSSM implies absence of zeros, but as discussed below Theorem~\ref{thm:main}, absence of zeros does not necessarily imply VSSM and neither does SSM.
So the question of whether SSM implies absence of zeros is still open.
A natural approach to this question would be to see if adding conditions on the graph class besides having the SSM property helps.
As an example in this direction, let us note that as a corollary to our proof of Theorem~\ref{thm:main} we directly obtain the following result.
\begin{cor}
Let $\Delta\in \mathbb{N}, C>0, \alpha\in (0,1)$ and $\lambda^\star>0$.
There exists $g_0=g_0(C,\alpha,\Delta,\lambda^\star)$ such that for any family $\mathcal{G}$ of graphs of maximum degree at most $\Delta$ and girth at least $g_0$, if the hard-core model on $\mathcal{G}$ satisfies strong spatial mixing at $\lambda^\star$ with parameters $C,\alpha$, then there exists $\varepsilon>0$ such that for each $G=(V,E)\in \mathcal{G}$ and $\Vec\lambda\in B(\lambda^\star,\varepsilon)^V$ we have $Z_G(\Vec \lambda)\neq 0$.
\end{cor}

Finally, let us remark that our methods for proving Theorem~\ref{thm:main} are fairly general and we believe they should extend to other $2$-state model such as the Ising model, where of course the appropriate modifications should be made taking boundary conditions into account.
For models with more than $2$ states, such as the Potts model, it is less clear how to even define VSSM, since there does not appear to be a tree-like object that captures the ratios as in the case of $2$-state systems~\cite{liu2024counterexamples}. 

\bibliographystyle{alpha}
\bibliography{graphpolynomials}



\end{document}